\documentclass{amsart}
\usepackage{latexsym,amsxtra,amscd,ifthen}
\usepackage{amsfonts}
\usepackage{verbatim}
\usepackage{amsmath}
\usepackage{amsthm}
\usepackage{amssymb}
\usepackage{color}
\usepackage{xcolor}
\usepackage{hyperref}
\usepackage{cancel}
\hypersetup{
 colorlinks  = true,  
 urlcolor   = black,  
 linkcolor  = black,  
 citecolor  = black   
}

\definecolor{green}{RGB}{20,140,10}

\numberwithin{equation}{section}

\mathchardef\mhyphen="2D

\theoremstyle{plain}
\newtheorem{theorem}{Theorem}[section]
\newtheorem*{theorem*}{Theorem}
\newtheorem{lemma}[theorem]{Lemma}
\newtheorem{proposition}[theorem]{Proposition}
\newtheorem{hypothesis}[theorem]{Hypothesis}
\newtheorem{corollary}[theorem]{Corollary}

\theoremstyle{definition}
\newtheorem{definition}[theorem]{Definition}
\newtheorem{example}[theorem]{Example}

\newtheorem{remark}[theorem]{Remark}
\newtheorem{question}[theorem]{Question}

\makeatletter              % This sequence of commands will
\let\c@equation\c@theorem  % incorporate equation numbering
\makeatother

\newcommand{\bfl}{\mathfrak l}

\DeclareMathOperator{\ged}{ged}
\DeclareMathOperator{\gldim}{gldim}

\DeclareMathOperator{\Ext}{Ext} 
\DeclareMathOperator{\Tor}{Tor}
\DeclareMathOperator{\R}{R}
\DeclareMathOperator{\pdim}{pdim}

 \DeclareMathOperator{\cd}{cd}
\DeclareMathOperator{\lcd}{lcd} 
 
 \DeclareMathOperator{\Gr}{\mhyphen Gr}

\DeclareMathOperator{\rate}{rate}
\DeclareMathOperator{\slope}{slope}
\DeclareMathOperator{\amp}{amp}

\DeclareMathOperator{\injdim}{injdim}

\DeclareMathOperator{\ASregxi}{ASreg_{\xi}} 
\DeclareMathOperator{\ASreg}{ASreg} 
\DeclareMathOperator{\CMreg}{CMreg} 
\DeclareMathOperator{\Extreg}{Extreg} 
\DeclareMathOperator{\Torreg}{Torreg}
\DeclareMathOperator{\CMregxi}{CMreg_{\xi}} 
\DeclareMathOperator{\Extregxi}{Extreg_{\xi}} 
\DeclareMathOperator{\Torregxi}{Torreg_{\xi}}
\DeclareMathOperator{\degxi}{deg_{\xi}}
\DeclareMathOperator{\gedxi}{ged_{\xi}}

\DeclareMathOperator{\reg}{reg}
\DeclareMathOperator{\D}{\mathsf{D}} \DeclareMathOperator{\Hom}{Hom}
\DeclareMathOperator{\RHom}{RHom}

\DeclareMathOperator{\depth}{depth} \DeclareMathOperator{\im}{im}

\newcommand{\fm}{\mathfrak{m}}

\newcommand{\be}{\begin{enumerate}}
\newcommand{\ee}{\end{enumerate}}
\newcommand{\bq}{\begin{eqnarray*}}
\newcommand{\eq}{\end{eqnarray*}}
\newcommand{\bqn}{\begin{eqnarray}}
\newcommand{\eqn}{\end{eqnarray}}
\newcommand{\op}{\textnormal{op}}
\newcommand{\fg}{\textnormal{fg}}
\newcommand{\bb}{\textnormal{b}}

\begin{document}
\title{Weighted homological regularities}

\author{E. Kirkman, R. Won and J.J. Zhang}

\address{Kirkman: Department of Mathematics,
P. O. Box 7388, Wake Forest University, Winston--Salem, NC 27109, USA}
\email{kirkman@wfu.edu}

\address{Won: Department of Mathematics, The George Washington
University, Washington, DC 20052, USA}
\email{robertwon@gwu.edu}

\address{Zhang: Department of Mathematics, Box 354350,
University of Washington, Seattle, WA 98195, USA}
\email{zhang@math.washington.edu}

\begin{abstract} 
Let $A$ be a noetherian connected graded algebra. We introduce 
and study homological invariants that are weighted sums of the 
homological and internal degrees of cochain complexes of graded 
$A$-modules, providing weighted versions of Castelnuovo--Mumford 
regularity, Tor-regularity, Artin--Schelter regularity, and 
concavity. In some cases an invariant (such as Tor-regularity) 
that is infinite can be replaced with a weighted invariant that 
is finite, and several homological invariants of complexes can 
be expressed as weighted homological regularities. We prove a 
few weighted homological identities some of which unify different 
classical homological identities and produce interesting new ones. 
\end{abstract}

\makeatletter
\@namedef{subjclassname@2020}{%
  \textup{2020} Mathematics Subject Classification}
\makeatother
\subjclass[2020]{16E10, 16E65, 20J99}

\keywords{Artin--Schelter regular algebra, Castelnuovo--Mumford 
regularity, Tor-regularity, Koszul algebra, Artin--Schelter 
regularity, concavity}

%\thanks{the US National Science Foundation.}

\maketitle

%\tableofcontents

\setcounter{section}{-1}

\section{Introduction}
\label{xxsec0}

Let $\Bbbk$ be a base field and let $A$ be a connected graded 
$\Bbbk$-algebra. If $X$ is a complex of graded left $A$-modules, 
then there are two natural gradings on $X$, namely, the gradings 
by homological and internal degrees. Properties of $A$ can be 
reflected in the relationships between these degrees. For example, 
$A$ is Koszul if the trivial graded $A$-module $\Bbbk$ has a 
minimal free resolution of the form
\begin{equation}
\label{E0.0.1}\tag{E0.0.1}
\cdots \to A(-i)^{\beta_i}\to A(-i+1)^{\beta_{i-1}}
\to \cdots \to A(-1)^{\beta_1} \to A\to \Bbbk\to 0,
\end{equation}
or equivalently, $\Tor^A_i(\Bbbk,\Bbbk)_{j}=0$ for all $j\neq i$. 
In this case we say that the trivial $A$-module $\Bbbk$ has a 
\emph{linear resolution}, or that the \emph{Tor-regularity of 
$\Bbbk$} is 0. The Tor-regularity of a complex $X$ is defined 
to be
\begin{equation}
\label{E0.0.2}\tag{E0.0.2}
\Torreg(X)=\sup_{i,j \in \mathbb{Z}}\{j-i\mid 
\Tor^A_i(\Bbbk, X)_j\neq 0 \}.
\end{equation} 
This supremum of a particular linear combination of internal and 
homological degrees provides a measure of the growth of the degrees 
of generators of the free modules in a minimal free resolution of 
$X$. 

When $A$ is a noetherian commutative graded algebra generated in 
degree one, the Tor-regularity of the trivial module is either 
zero or infinity \cite{AP}, so the Tor-regularity measures only 
whether $A$ is Koszul or not. J{\o}rgensen and Dong--Wu 
\cite{Jo3, Jo4, DW} studied the Tor- and $\Ext$-regularities for 
noncommutative algebras, and further results on these regularities 
were the subject of \cite{KWZ2}. As shown in 
\cite[Example 2.4(4)]{KWZ2}, for any non-negative integer $n$, 
there is a noncommutative algebra whose trivial graded module 
$\Bbbk$ has Tor-regularity $n$, and so the Tor-regularity 
of the trivial graded module of a noncommutative algebra provides 
more information than whether or not the algebra is Koszul.

Castelnuovo--Mumford regularity (CM regularity for short) was 
introduced for commutative graded algebras as the supremum of 
another linear combination of homological and internal degrees 
(in this case the homological degree involves local cohomology).
Over a noetherian commutative graded algebra, every finitely 
generated graded module has finite CM regularity. In the 
noncommutative case, CM regularity was studied by J{\o}rgensen 
and Dong--Wu \cite{Jo3, Jo4, DW} and explored further in 
\cite{KWZ2}; for a noncommutative noetherian algebra, some 
finitely generated modules may have infinite CM regularity 
\cite[Example 5.1]{KWZ2}. In \cite[Definition 0.6]{KWZ2} a new 
notion of regularity involving both internal and homological 
degrees, the Artin--Schelter regularity, was introduced; it 
measures how close an algebra is to being an Artin-Schelter 
regular algebra [Definition \ref{xxdef0.1}].

In this paper we introduce {\em weighted} versions of classical 
homological invariants such as the $\Tor$-, $\Ext$- and CM 
regularities, as well as a weighted version of the 
Artin--Schelter regularity. These weighted invariants are defined 
as extrema of general weighted sums of homological and internal 
degrees of certain complexes, and hence they extend the original 
version of these invariants. Moreover, we will see that other 
useful invariants, such as the $\sup$, $\inf$, projective 
dimension and depth of a complex, can be viewed in the context 
of weighted regularities of complexes. These weighted invariants 
can provide new finite invariants, even for commutative algebras. 
For example, Proposition~\ref{xxpro5.8} gives a condition which 
guarantees the existence of some weight such that a 
\emph{weighted} Tor-regularity [Definition~\ref{xxdef0.2}] of 
the trivial module is finite. In particular, for a noetherian 
commutative algebra it implies that such a weight always exists.

We now define the weighted regularities that will be the focus 
of this paper. An $\mathbb{N}$-graded algebra $A$ is called 
{\it connected graded} if $A_0=\Bbbk$. For a connected graded 
algebra $A$, let $\fm = A_{\geq 1}$ and $\Bbbk = A/\fm$. An 
important class of connected graded algebras in this paper are 
the {\it Artin--Schelter regular} algebras \cite{AS} which play 
a central role in noncommutative algebraic geometry and 
representation theory.

\begin{definition}[{\cite[p.171]{AS}}]
\label{xxdef0.1}
A connected graded algebra $T$ is called 
{\it Artin--Schelter Gorenstein} (or {\it AS Gorenstein}, 
for short) if the following conditions hold:
\begin{enumerate}
\item[(a)]
$T$ has injective dimension $d<\infty$ on the left and on the 
right,
\item[(b)]
$\Ext^i_T({}_T\Bbbk, {}_{T}T)=\Ext^i_{T}(\Bbbk_T,T_T)=0$ for 
all $i\neq d$, and
\item[(c)]
$\Ext^d_T({}_T\Bbbk, {}_{T}T)\cong \Ext^d_{T}(\Bbbk_T,T_T)
\cong \Bbbk(\bfl)$ for some integer $\bfl$. Here $\bfl$ is 
called the {\it AS index} of $T$.
\end{enumerate}
In this case, we say $T$ is of type $(d,\bfl)$. If in addition,
\begin{enumerate}
\item[(d)]
$T$ has finite global dimension, and
\item[(e)]
$T$ has finite Gelfand--Kirillov dimension,
\end{enumerate}
then $T$ is called {\it Artin--Schelter regular} (or {\it AS
regular}, for short) of dimension $d$.
\end{definition}

In this paper we generally reserve the letters $S$ and $T$ for 
AS regular algebras. Note that the only commutative AS regular 
algebras are polynomial rings and so AS regular algebras are 
regarded as noncommutative versions of commutative polynomial 
rings. Recall that the $i$th local cohomology of a graded left 
$A$-module $M$ is defined to be
\[
H^i_{\fm}(M)=\lim\limits_{n\to \infty} \Ext^i_A(A/\fm^n, M).
\]
If $X$ is a complex of graded left $A$-modules, one can define 
the $i$th local cohomology of $X$, denoted by $H^i_{\fm}(X)$, 
similarly, as in \cite{Jo1, Jo4}.

\begin{definition}
\label{xxdef0.2}
Fix a real number $\xi$. Let $A$ be a noetherian connected graded 
algebra and let $X$ be a nonzero complex of graded left $A$-modules. 
\begin{enumerate}
\item[(1)]
The {\it $\xi$-Tor-regularity} of $X$ is defined to be
\[
\Torregxi (X)=\sup_{i, j \in \mathbb{Z}}\{ j -\xi i \; \mid \;
\Tor^A_i(\Bbbk, X)_j\neq 0\}.
\]
If $\xi=1$, then $\Torregxi(X)$ agrees with the usual 
Tor-regularity $\Torreg(X)$ defined in \eqref{E0.0.2} 
\cite{Jo3, Jo4, DW}.
\item[(2)]
The {\it $\xi$-Castelnuovo--Mumford regularity} (or 
{\it $\xi$-CM regularity}, for short) of $X$ is defined to be
\[
\CMregxi (X)=\sup_{i,j \in \mathbb{Z}} \{j +\xi i \; \mid \; 
H^{i}_{\fm}(X)_{j}\neq 0 \}.
\]
If $\xi=1$, then $\CMregxi(X)$ agrees with the usual 
Castelnuovo--Mumford regularity $\CMreg(X)$ defined in 
\cite{Jo3, Jo4, DW}.
\item[(3)]
The {\it $\xi$-Artin--Schelter regularity} (or 
{\it $\xi$-AS regularity}) of $A$ is defined to be 
\[
\ASregxi(A)=\Torregxi(\Bbbk)+\CMregxi(A).
\]
If $\xi = 1$, then $\ASregxi(A)$ agrees with the AS regularity 
introduced in \cite{KWZ2}.
\end{enumerate}
\end{definition}

The notions of regularity given in parts (1) and (2) above
are natural generalizations of the classical Tor- and 
Castelnuovo--Mumford regularities \cite{Mu3, Ei, EiG}. These 
weighted homological invariants provide useful information 
about the graded algebra $A$ and graded modules (or complexes 
of modules) over $A$. In Section~\ref{xxsec2} we will extend 
these definitions of weighted regularities to consider weights 
of the form $\xi = (\xi_0, \xi_1)$ and more general linear 
combinations of homological and internal degrees 
[Definitions \ref{xxdef2.1}, \ref{xxdef2.3}, and \ref{xxdef2.5}]. 
For simplicity in this introduction we consider only the case 
where $\xi_0 = 1$ and $\xi_1= \xi$. 

For a finitely generated graded $A$-module $M$, the relations 
between the regularities $\Torreg(M)$ and $\CMreg(M)$ have 
been studied in the literature. When $A$ is a polynomial ring 
generated in degree 1, $\Torreg(M) = \CMreg(M)$ \cite{EiG}, but 
this is not the case for all AS regular algebras 
\cite[Theorem 5.4]{DW}, see also Theorem~\ref{xxthm4.3}. Other 
relations between these invariants were established in the 
commutative case \cite{Rom} and were extended to the 
noncommutative case in \cite{Jo3, Jo4, DW}. In this paper we 
provide further relations between the weighted versions of these 
invariants in the noncommutative case. The following two theorems 
extend \cite[Theorems 2.5 and 2.6]{Jo4}, \cite[Theorem 4.2]{Rom},
and \cite[Proposition 5.6]{DW}.

\begin{theorem}
\label{xxthm0.3}
Let $A$ be a noetherian connected graded algebra with balanced
dualizing complex. Let $X$ be a nonzero object in 
$\D^{\bb}_{\fg}(A\Gr)$ and let $\xi \in \mathbb{R}$.
\begin{enumerate}
\item[(1)] {\rm (Theorem \ref{xxthm3.3})}
\[
\Torregxi(X)\leq \CMregxi(X)+\Torregxi(\Bbbk).
\]
\item[(2)] {\rm (Theorem \ref{xxthm3.5})}
\[
\CMregxi(X)\leq \Torregxi(X)+\CMregxi(A).
\]
\item[(3)] 
\[
\ASregxi(A)\geq 0.
\]
\end{enumerate}
\end{theorem}

\begin{theorem}[Theorem \ref{xxthm3.10}]
\label{xxthm0.4}
Let $A$ be a noetherian connected graded algebra with balanced
dualizing complex. Let $X$ be a nonzero object in 
$\D^{\bb}_{\fg}(A\Gr)$ of finite projective dimension.
\begin{enumerate}
\item[(1)] 
Suppose $0\leq \xi\leq 1$. Then
\[
\CMregxi(X)= \Torregxi(X)+\CMregxi(A).
\]
\item[(2)] 
For all $\xi \ll 0$,
\[
\CMregxi(X)= \Torregxi(X)+\CMregxi(A).
\]
\end{enumerate}
\end{theorem}

If $M$ is a graded vector space, let $\deg(M)$ denote the 
maximal degree of the nonzero homogeneous elements in $M$,
as in equation~\eqref{E1.1.1}. (A more general two-parameter 
definition of the weighted degree of a complex is given 
in equation~\eqref{E1.1.3}.)

\begin{remark}
\label{xxrem0.5}
Retain the hypotheses of Theorem \ref{xxthm0.4}.
\begin{enumerate}
\item[(1)]
If $\xi>1$, then by Remarks~\ref{xxrem2.6}(1) and 
\ref{xxrem3.11}(1), the conclusion of Theorem \ref{xxthm0.4} 
may fail to hold, even when $A$ is AS regular and $X=\Bbbk$.
\item[(2)]
It is unknown if Theorem \ref{xxthm0.4}(2) holds for all
$\xi<0$, see Remark \ref{xxrem3.11}(2). 
\item[(3)]
The famous Auslander--Buchsbaum formula in the graded setting 
can be recovered from by taking $\lim\limits_{\xi\to -\infty} 
\frac{1}{\xi} (-)$ in Theorem \ref{xxthm0.4}(2) (see Corollary 
\ref{xxcor3.12}(1)). Hence Theorem \ref{xxthm0.4} unifies the 
Auslander--Buchsbaum formula \cite[Theorem 3.2]{Jo2} with 
\cite[Theorem 4.2]{Rom} (in the commutative case) and 
\cite[Proposition 5.6]{DW} and \cite[Theorem 0.7]{KWZ2} 
(in the noncommutative case).
\item[(4)]
In addition to part (3), when taking 
$\lim\limits_{\xi\to -\infty}(-)$ of Theorem \ref{xxthm0.4}(2), 
in Corollary~\ref{xxcor3.12}(2), we obtain a new homological 
identity 
\begin{equation}
\label{E0.5.1}\tag{E0.5.1}
\deg H^{d(X)}_{\fm}(X)
=\deg \Tor^A_{p(X)}(\Bbbk, X)+\deg H^{d(A)}_{\fm}(A),
\end{equation}
where $p(X):= \pdim (X)$ and $d(X):=\depth(X)$. We call 
\eqref{E0.5.1} a \emph{refined Auslander--Buchsbaum formula}.
\end{enumerate}
\end{remark}

In \cite[Theorem 0.8]{KWZ2}, we generalized a result of Dong and 
Wu \cite[Theorems 4.10 and 5.4]{DW} to the not-necessarily Koszul 
setting to show that a noetherian connected graded algebra $A$ 
with balanced dualizing complex is AS regular if and only if 
$\ASreg(A) = 0$. Here, we extend this result to the weighted 
setting. The Cohen--Macaulay property will be defined in 
Definition \ref{xxdef1.2}.

\begin{theorem}
\label{xxthm0.6}
Let $A$ be a noetherian connected graded algebra with balanced 
dualizing complex. Then the following are equivalent:
\begin{enumerate}
\item[(1)]
$A$ is AS regular.
\item[(2)]
There exists a $\xi \leq 1$ such that $\ASregxi(A)=0$. 
\item[(3)]
$A$ is Cohen--Macaulay and there exists a $\xi \in \mathbb{R}$ 
such that $\ASregxi(A)=0$.
\end{enumerate}
\end{theorem}

It is an open question if the hypothesis that $\xi \leq 1$ can 
be removed from part (2) or the hypothesis that $A$ is 
Cohen--Macaulay can be removed from part (3).

By Example \ref{xxex2.2}(3), if $T$ is AS regular (or AS 
Gorenstein) of type $(d,\bfl)$, then
\begin{equation}
\label{E0.6.1}\tag{E0.6.1}
\CMregxi(T)=\xi d-\bfl
\end{equation}
which will appear in several places in this paper. If $T$ is 
AS regular and $\xi\leq 1$, then we also have
$$\Torregxi(\Bbbk)=-\xi d+\bfl=-\CMregxi(T),$$
which follows from a direct computation, or from \eqref{E0.6.1} 
and the fact $\ASregxi(A)=0$. Some further computations of 
weighted regularities in the non-Koszul case are provided in 
Remark \ref{xxrem2.6}. 

As an immediate consequence of Theorem \ref{xxthm0.6}, we have
the following corollary.

\begin{corollary}
\label{xxcor0.7}
Let $A$ be a noetherian AS Gorenstein algebra of type $(d,\bfl)$.
Suppose there is a $\xi\in {\mathbb R}$ such that 
$$\deg \Tor^A_i(\Bbbk,\Bbbk)\leq \xi i + \bfl-\xi d$$
for all $i\geq 0$. Then $A$ is AS regular.
\end{corollary}

Note that Corollary \ref{xxcor0.7} recovers 
\cite[Theorem 4.10]{DW} by setting $\bfl=d$ and $\xi=1$.
We also prove a weighted version of \cite[Corollary 5.2]{Jo3}, 
which can be used to compute the weighted CM regularity of a 
finitely generated graded module over an AS Gorenstein algebra.

\begin{theorem}[Theorem \ref{xxthm4.6}]
\label{xxthm0.8} 
Suppose $A$ is a noetherian AS Gorenstein algebra of 
type $(d,\bfl)$. Let $\xi \leq 1$ be a real number and
let $M\neq 0$ be a finitely generated graded left 
$A$-module with finite projective dimension.
\begin{enumerate}
\item[(1)]
Let $w$ be an integer with $0 \leq w \leq d$. Then
$$\begin{aligned}
\max_{0 \leq j \leq w}\{\deg H^j_{\fm}(M) & +\xi j\}
=-\bfl+\xi d+\max_{d-w \leq j \leq d}
\{\deg \Tor^A_{j}(\Bbbk, M) -\xi j\}.
\end{aligned}
$$
\item[(2)]
In particular, if $w$ is chosen to be maximal with the 
property that $H^w_{\fm}(M)\neq 0$, we have
$$\CMregxi(M)=-\bfl+\xi d+\max_{d-w \leq j \leq d}
\{\deg \Tor^A_j(\Bbbk,M)-\xi j\}.$$
\item[(3)] 
If, further, $M$ is $s$-Cohen--Macaulay, then $p:=d-s$ is the 
projective dimension of $M$ and 
$$\CMregxi(M)=-\bfl+\xi s+\deg(\Tor^A_{p}(\Bbbk, M)).$$
\end{enumerate}
\end{theorem}

The above results are generalizations of classical results in 
various directions, for example, from commutative algebras to 
noncommutative algebras, from modules to complexes, and from 
unweighted regularities to weighted regularities.

Homological invariants, as well as homological identities have  
many applications. In \cite[Theorems 0.8 and 0.10]{KWZ1}, the 
authors used regularities to bound the maximal degree of 
generators of the invariant rings under Hopf actions. In 
\cite[Corollary 0.11]{KWZ2}, the authors demonstrated how to 
control the Koszul property of an AS regular algebra $A$ by 
using a finite map $T \to A$. Following the ideas of Backelin 
\cite{Ba}, we can also use finite maps to control the Koszul 
property of higher Veronese subrings.

In Section~\ref{xxsec5} we consider the case when there is a 
finite graded map from a noetherian AS regular algebra $T$ into 
a connected graded algebra $A$. In Proposition \ref{xxpro5.8} 
we show there exists a weight $\xi$ with $\Torregxi(X)<\infty$ 
for all $X\in \D^{\bb}_{\fg}(A\Gr)$. This result can be applied 
to all noetherian commutative graded algebras, where the 
Tor-regularity is not always finite. As the value of 
$\Torreg({}_A\Bbbk)$ is related to the Koszul property of $A$, 
the values of $\Torregxi({}_A \Bbbk)$ are related to the Koszul 
property of the Veronese subrings of $A$ \cite[Corollary, p. 81]{Ba}.
See a related result in Proposition \ref{xxpro5.8}(1).

\begin{corollary}[Corollary \ref{xxcor5.9}]
\label{xxcor0.11}
Let $A$ be a noetherian algebra generated in degree 1 and suppose 
there is a finite map $T \to A$ where $T$ is a noetherian 
connected graded algebra of finite global dimension. Then 
$A^{(d)}$ is Koszul for $d\gg 0$.
\end{corollary}

When $A$ is commutative, Corollary \ref{xxcor5.9} recovers a 
very nice result of Mumford \cite[Theorem 1]{Mu2}. If we remove 
the hypothesis that there is a finite map from a noetherian 
connected graded algebra $T$ of finite global dimension to $A$, 
it is an open question if the conclusion of Corollary 
\ref{xxcor0.11} holds, see Question \ref{xxque5.10}. 

The paper is organized as follows.  Section~\ref{xxsec1} recalls 
some basic definitions and properties of homological algebra 
(including local cohomology). In Section~\ref{xxsec2} we provide
the full definitions of the weighted regularities defined above 
in Definition~\ref{xxdef0.2}, and explicitly compute these 
invariants in several examples. Section~\ref{xxsec3} proves 
generalizations of some equalities and inequalities from 
\cite{Jo4, DW} that comprise Theorems \ref{xxthm0.3} and 
\ref{xxthm0.4}. Theorems \ref{xxthm0.6} and \ref{xxthm0.8} on 
weighted AS regularities and the computation of weighted CM 
regularities of graded modules over AS Gorenstein algebras are 
proved in Section~\ref{xxsec4}. In Section~\ref{xxsec5} we 
provide some comments and remarks about related homological 
invariants such as concavity, rate, and slope. 

\section{Preliminaries}
\label{xxsec1}
For an $\mathbb{N}$-graded $\Bbbk$-algebra $A$, we let $A \Gr$ 
denote the category of $\mathbb{Z}$-graded left $A$-modules.
When convenient, we identify the category of graded right 
$A$-modules with $A^{\op}\Gr$ where $A^{\op}$ is the opposite
ring of $A$. The derived category of graded $A$-modules is 
denoted by $\D(A \Gr)$. We use the standard notation 
$\D^+(A \Gr)$, $\D^-(A \Gr)$, and $\D^{\bb}(A \Gr)$ for the 
full subcategories of (cochain) complexes $X=(X^n)$ of 
$\mathbb{Z}$-graded left $A$-modules which are 
bounded below (i.e., $X^n=0$ for all $n\ll 0$), 
bounded above (i.e., $X^n=0$ for all $n\gg 0$), and 
bounded (i.e., $X^n=0$ for all $|n|\gg 0$), respectively. 
We use the 
subscript $\fg$ to denote the full subcategories consisting 
of complexes with finitely generated cohomology, e.g., 
$\D^{\bb}_{\fg}(A \Gr)$. We adopt the standard convention 
that a left $A$-module $M$ can be viewed as a complex 
concentrated in position $0$.

Let $\ell$ be an integer. For a graded $A$-module $M$, the 
shifted $A$-module $M(\ell)$ is defined by
\[ M(\ell)_m = M_{m + \ell}\]
for all $m \in \mathbb{Z}$. For a cochain complex $X = 
(X^n, \;d_X^n: X^n \to X^{n+1})$, we define two notions of 
shifting: $X(\ell)$ shifts the degrees of each graded vector 
space $X(\ell)^i_m = X^i_{m+\ell}$ (together with differential $d^i_{X(\ell)}
=d^i_X$) for all $i, m \in 
\mathbb{Z}$ and $X[\ell]$ shifts the complex $X[\ell]^i = 
X^{i+\ell}$ (together with differential $d^i_{X[\ell]}=(-1)^{\ell}d^i_X$)  
for all $i \in \mathbb{Z}$.

Let $M=\bigoplus_{d\in {\mathbb Z}} M_d$ be a 
${\mathbb Z}$-graded $\Bbbk$-vector space. We say that $M$ 
is {\em locally finite} if $\dim_{\Bbbk} M_d<\infty$ for all 
$d\in {\mathbb Z}$. 

\begin{definition}
\label{xxdef1.1}
Let $A:=\bigoplus_{i\geq 0} A_i$ be a locally finite 
$\mathbb{N}$-graded algebra. The {\it Hilbert series of $A$} 
is defined to be
\begin{equation*}
h_A(t)=\sum_{i\in {\mathbb N}} (\dim_{\Bbbk} A_i)t^i.
\end{equation*}
Similarly, if $M = \bigoplus_{i \in \mathbb{Z}} M_i$ is a 
locally finite $\mathbb{Z}$-graded $A$-module (or 
$\mathbb{Z}$-graded vector space), the \emph{Hilbert series 
of $M$} is defined to be
\begin{equation*}
h_M(t)=\sum_{i\in \mathbb{Z}} (\dim_{\Bbbk} M_i)t^i.
\end{equation*}
\end{definition}

Define the {\it degree} of $M$ to be the maximal degree of 
the nonzero homogeneous elements in $M$, namely, 
\begin{equation}
\label{E1.1.1}\tag{E1.1.1}
\deg(M)=\inf\{d\mid  (M)_{\geq d}= 0\}-1=
\sup\{d\mid  (M)_{d}\neq 0\} \quad \in \quad  
{\mathbb Z} \cup\{\pm \infty\}.
\end{equation}
By convention, we define $\deg (0)=-\infty$. Similarly, we 
define
\begin{equation}
\label{E1.1.2}\tag{E1.1.2}
\ged(M)=\sup\{d\mid  (M)_{\leq d}= 0\}+1=
\inf\{d\mid  (M)_{d}\neq 0\} \quad \in \quad  
{\mathbb Z} \cup\{\pm \infty\}.
\end{equation}
By convention, we define $\ged (0)=\infty$. Now we fix 
$\xi:=(\xi_0,\xi_1)$ a pair of real numbers not both  zero and 
move from the case $(1,\xi)$ considered in the introduction to 
the general case $\xi = (\xi_0,\xi_1)$. We will see in 
Lemma~\ref{xxlem3.1}(1) that if $\xi_0 > 0$, then we can often 
rescale so that $\xi_0 = 1$. For a cochain complex $X$,
the \emph{$\xi$-weighted degree} of $X$ is defined to be
\begin{equation}
\label{E1.1.3}\tag{E1.1.3}
\degxi(X) =\sup_{m,n \in \mathbb{Z}} \{ \xi_0 m+ \xi_1 n \mid   
H^n(X)_m\neq 0\}.
\end{equation}
Similarly, the \emph{$\xi$-weighted $\ged$} of $X$ is defined to 
be
\begin{equation}
\label{E1.1.4}\tag{E1.1.4}
\gedxi(X) =\inf_{m,n \in \mathbb{Z}}
\{ \xi_0 m+ \xi_1 n \mid H^n(X)_m\neq 0\}.
\end{equation}
Recall that the \emph{supremum} and \emph{infimum} of $X$ are defined to be
\[
\sup(X) = \sup_{n \in \mathbb{Z}} \{ n \mid H^n(X) \neq 0\} \quad \text{and} \quad \inf(X) = \inf_{n \in \mathbb{Z}} \{n \mid H^n(X) \neq 0\}.
\]
It is clear that, if $X$ is nonzero in $\D(A\Gr)$, then
\[
\sup(X) = \deg_{(0,1)} (X)  \quad \text{and} \quad \inf(X) = \ged_{(0,1)} (X).
\]
Further,
\[
-\inf (X)=\deg_{(0,-1)} (X) \quad \text{ and } \quad 
-\sup (X)=\ged_{(0,-1)} (X).
\]

If
$$X = \hspace{.2in} \cdots \rightarrow X^{s-1} \rightarrow 
X^s\rightarrow X^{s+1} \rightarrow \cdots,$$
then the brutal truncations of $X$ are denoted by
$$X^{\geq s} := \hspace{.2in} \cdots \to 0\to \cdots \to 0 
\rightarrow X^{s} \rightarrow X^{s+1} \rightarrow \cdots $$
and 
$$X^{\leq s}:= \hspace{.2in}  \cdots \rightarrow X^{s-1} 
\rightarrow X^{s} \rightarrow   0 \to \cdots \to 0\to \cdots.$$

Let $A$ be a connected graded algebra with graded Jacobson 
radical $\fm:=A_{\geq 1}$. Let $\Bbbk$ also denote the graded 
$A$-bimodule $A/\fm$. For a graded left $A$-module $M$, let
\begin{equation}
\label{E1.1.5}\tag{E1.1.5}
t^A_i(_A M)=\deg \Tor^A_i(\Bbbk, M).
\end{equation}
If $M$ is a graded right $A$-module, let
\begin{equation}
\label{E1.1.6}\tag{E1.1.6}
t^A_i(M_A)=\deg \Tor^A_i(M,\Bbbk).
\end{equation}
It is clear that $t^A_i(_A \Bbbk)=t^A_i(\Bbbk_A)$.
If the context is clear, we will use $t^A_i(M)$
instead of $t^A_i(_A M)$ (or $t^A_i(M_A)$). 

For each graded left $A$-module $M$, we define
$$\Gamma_{\fm}(M) 
=\{ x\in M\mid A_{\geq n} x=0 \; {\text{for some $n\geq 1$}}\;\}
=\lim_{n\to \infty} \Hom_A(A/A_{\geq n}, M)$$
and call this the {\it $\fm$-torsion submodule} of $M$. It is 
standard that the functor $\Gamma_{\fm}(-)$ is a left exact 
functor $A\Gr \to A\Gr$. Since the category $A \Gr$ has enough 
injectives, the $i$th right derived functors, denoted by 
$H^i_{\fm}$ or $\R^i\Gamma_{\fm}$, are defined and called the 
{\it local cohomology functors}. Explicitly, one has 
$$H^i_{\fm}(M)=\R^i\Gamma_{\fm}(M)
:=\lim_{n\to \infty} \Ext^i_A(A/A_{\geq n}, M).$$ 
See \cite{AZ, VdB} for more details. For a complex $X$ of 
graded left $A$-modules, the local cohomology functors are 
defined by  
$$H^i_{\fm}(X)=\R^i\Gamma_{\fm}(X)
:=\lim_{n\to \infty} \Ext^i_A(A/A_{\geq n}, X).$$ 

\begin{definition}
\label{xxdef1.2} 
Let $A$ be a connected graded noetherian algebra. Let $M$ be a 
finitely generated graded left  $A$-module. We call $M$ 
{\it $s$-Cohen--Macaulay} or simply {\it Cohen--Macaulay} if 
$H^i_{\fm}(M) = 0$ for all $i \neq s$ and $H^s_{\fm}(M) \neq 0$. 
We say $A$ is {\it Cohen--Macaulay} if $_AA$ is Cohen--Macaulay.
\end{definition}

Throughout the rest of this paper, we assume the following 
hypothesis; we refer the reader to \cite{Ye} for the definitions 
of a dualizing complex and a balanced dualizing complex.  

\begin{hypothesis}
\label{xxhyp1.3}
Let $A$ be a noetherian connected graded algebra with balanced 
dualizing complex. In this case by \cite[Theorem 6.3]{VdB} the 
balanced dualizing complex will be given by $\R\Gamma_\fm(A)'$, 
where $'$ represents the graded vector space dual.
\end{hypothesis}

The local cohomological dimension of a graded $A$-module 
$M$ is defined to be
$$\lcd(M) :=\sup\{i \in \mathbb{Z} \mid 
H^i_{\mathfrak{m}}(M) \neq 0\}$$
and the cohomological dimension of $\Gamma_\fm$ is defined 
to be
$$\cd(\Gamma_\fm) = \sup_{M \in A \Gr} \{\lcd(M)\}.$$

We will use the following {\it Local Duality Theorem} of 
Van den Bergh several times.

\begin{theorem} 
\label{xxthm1.4}
Let $A$ be a noetherian connected graded $\Bbbk$-algebra with
$\cd(\Gamma_\fm)< \infty$ and let $C$ be a connected graded 
algebra. 
\begin{enumerate}
\item[(1)]\cite[Theorem 5.1]{VdB} 
For any $X \in \D((A\otimes C^{\op})\Gr)$ there is an isomorphism
$$\R\Gamma_\fm(X)' \cong \RHom_A(X,\R\Gamma_\fm(A)')$$
in $\D((C\otimes A^{\op})\Gr)$.
\item[(2)] \cite[Observation 2.3]{Jo4}
Assume Hypothesis \ref{xxhyp1.3}.
If $X\in \D^{\bb}_{\fg}(A\Gr)$, then $\R\Gamma_\fm(X)'
\in \D^{\bb}_{\fg}(A^{\op}\Gr)$.
\end{enumerate}
\end{theorem}

\begin{proof} (2) By \cite[Theorem 6.3]{VdB}, $R:=\R\Gamma_\fm(A)'$
is the balanced dualizing complex over $A$. By a basic property 
of a balanced dualizing complex \cite[Proposition 3.4]{Ye}, 
$\RHom_A(X,R)\in \D^{\bb}_{\fg}(A^{\op}\Gr)$. By part (1), 
$\R\Gamma_\fm(X)' \cong \RHom_A(X,R)$, and so the assertion
follows.
\end{proof}

\section{Weighted versions of homological regularities}
\label{xxsec2}
In this section we introduce weighted versions of several 
homological invariants in the noncommutative setting and 
provide some sample computations of these regularities.
Castelnuovo--Mumford regularity in the noncommutative 
setting was first studied by J{\o}rgensen in \cite{Jo3, Jo4} 
and later by Dong and Wu \cite{DW}. Throughout, we fix an 
ordered pair $\xi=(\xi_0,\xi_1)$ of real numbers. In the 
introduction, we identified $\xi$ with $(1,\xi)$.

\begin{definition}
\label{xxdef2.1}
Let $X$ be a nonzero object in $\D^{\bb}_{\fg}(A\Gr)$. 
\begin{enumerate}
\item[(1)] 
The {\it $\xi$-Castelnuovo--Mumford regularity} of $X$ is 
defined to be
$$\CMregxi (X)= \degxi (\R\Gamma_{\fm}(X)).$$
If $\xi_0\neq 0$, then it is clear that
$$\CMregxi (X)=\sup_{i \in \mathbb{Z}} 
\{\xi_0 \deg (H^{i}_{\fm}(X)) +\xi_1 i \}.$$
\item[(2)]
If $\xi=(1,1)$, then $\CMregxi(X)$ becomes the ordinary 
$\CMreg(X)$ as defined in \cite[Definition 2.1]{Jo4}.
\end{enumerate}
\end{definition}

By Theorem 
\ref{xxthm1.4}(2), if $0 \neq X \in \D^{\bb}_{\fg} (A \Gr)$ 
then $\R\Gamma_{\fm}(X) \ncong 0$ and $\R\Gamma_{\fm}(X)' \in  
\D^{\bb}_{\fg} (A^{\op} \Gr)$. It follows that, if 
$\xi_0\geq 0$, then $\CMregxi(X)$ is finite. However, if $A$ 
does not have a balanced dualizing complex then $\CMregxi(A)$ 
and $\CMregxi(X)$ can be infinite (e.g. \cite [Example 5.1]{KWZ2}).

\begin{example}
\label{xxex2.2} 
Suppose that $\xi_0 \geq 0$. 
\begin{enumerate}
\item[(1)]
If $M$ is a finite-dimensional nonzero graded left $A$-module, 
then 
$$\CMregxi(M)=\xi_0 \deg (M).$$
A more general case is considered in part (4).
\item[(2)] 
Let $A$ be an AS Gorenstein algebra of type $(d, \bfl)$. Then
$\CMregxi (A) =\xi_1 d- \xi_0 \bfl$. 
\item[(3)]
Let $A$ be an AS regular algebra of type $(d, \bfl)$. 
By \cite[Proposition 3.1]{StZ}, when regarded as a rational function, $\deg_t h_A(t) 
= -\bfl$. Hence,
$$\CMregxi (A) =\xi_1 d-\xi_0 \bfl
=\xi_1 \gldim A+ \xi_0 \deg_t  h_A(t).$$
\item[(4)]
If $M$ is $s$-Cohen--Macaulay, then, by definition,
$$\CMregxi (M)=\xi_1 s+ \xi_0 \deg(H^s_{\fm}(M)).$$
\item[(5)]
If $\xi=(0,1)$ and $X \in \D^{\bb}_{\fg} (A \Gr)$, then 
$$\CMreg_{(0,1)}(X)=\sup(\R\Gamma_{\fm}(X)).$$
\item[(6)]
Recall from \cite{Jo1} that the depth of a complex $X$ is 
defined to be
$$\depth(X):=\inf(\RHom_A(\Bbbk, X)).$$ 
By \cite[Lemma 2.6]{DW}, $\depth(X)=
\inf(\R\Gamma_{\fm}(X))$.
If $\xi=(0,-1)$ and $X \in \D^{\bb}_{\fg} (A \Gr)$, then 
$$\CMreg_{(0,-1)}(X)=-\inf(\R\Gamma_{\fm}(X))
=-\depth(X).$$
\end{enumerate}
\end{example}

\begin{definition}
\label{xxdef2.3}
Let $X$ be a nonzero object in $\D^{\bb}_{\fg}(A\Gr)$. The 
{\it $\xi$-Ext-regularity} of $X$ is defined to be
\begin{align*}
\Extregxi (X)&=-\gedxi (\R \Hom_A(X, \Bbbk))\\
&= - \inf_{i \in \mathbb{Z}} 
\{\xi_0 \ged (\Ext^i_A(X, \Bbbk))+\xi_1  i\},
\end{align*}
where the second equality holds when $\xi_0\neq 0$. 
The {\it Tor-regularity} of $X$ is defined to be
$$\begin{aligned}
\Torregxi (X)&=\degxi( \Bbbk \otimes^{L}_{A} X) \\
&=\sup_{i \in \mathbb{Z}}
\{\xi_0 \deg (\Tor^A_{-i}( \Bbbk, X)) +\xi_1 i \}\\
&=\sup_{i \in \mathbb{Z}}
\{\xi_0 \deg (\Tor^A_{i}( \Bbbk, X)) -\xi_1 i \}
\end{aligned}$$
where the second equality holds when $\xi_0\neq 0$.
\end{definition}

By \cite[Remark 4.5]{DW}, if $X$ has a finitely generated
minimal free resolution over $A$, then 
$\Extregxi(X)=\Torregxi(X)$.

\begin{example}
\label{xxex2.4}
The following examples are clear.
\begin{enumerate}
\item[(1)]
If $M \in A \Gr$ and $r=\Torregxi(M)$ and $\xi_0>0$, then 
$$t^A_i(_A M):=\deg(\Tor^A_i(\Bbbk, M))\leq 
\xi_0^{-1}(r+\xi_1 i)$$
for all $i$.
\item[(2)] 
$\Torregxi(A)=\Extregxi(A)=0$.
\item[(3)]
Suppose $X\in \D^{\bb}_{\fg}(A\Gr)$. If $\xi=(0, -1)$, 
then 
$$\Torreg_{(0,-1)}(X)=\pdim(X)$$
where $\pdim$ denotes the projective dimension.
\item[(4)]
Suppose $\xi_0>0$. 
Let $A$ be a Koszul algebra as defined in the introduction 
\eqref{E0.0.1}, and let $g = \gldim A$. By definition, 
$\deg \Tor^i_{A}(\Bbbk,\Bbbk)= i$ for all $0\leq i\leq g$ 
and $-\infty$ otherwise. This implies that
\begin{equation}
\label{E2.4.1}\tag{E2.4.1}
\Torregxi(_A\Bbbk)=\begin{cases}
0 & \xi_1\geq \xi_0\\
+\infty & {\text{$\xi_1 <\xi_0$ and $g=\infty$}},\\
g(\xi_0-\xi_1) & 
{\text{$\xi_1 <\xi_0$ and $g<\infty$}}.
\end{cases}
\end{equation}
\item[(5)]
Suppose $\xi_0>0$ and $\xi_1<\xi_0$. Let $g = \gldim A$. 
Since $\deg \Tor^i_{A}(\Bbbk,\Bbbk)\geq i$ for all $0\leq i
\leq g$, we obtain 
$$\Torregxi(_A\Bbbk)=\begin{cases}
+\infty & g=\infty,\\
\deg \Tor^g_{A}(\Bbbk,\Bbbk) \xi_0-g\xi_1 & g<\infty.
\end{cases}$$
\end{enumerate}
\end{example}

Using the weighted versions of the Tor (or Ext) and CM 
regularities, we define a weighted version of the 
Artin-Schelter regularity of a connected graded algebra $A$, 
extending the unweighted ($\xi = (1,1)$) version, which was 
introduced in \cite{KWZ2}.

\begin{definition}
\label{xxdef2.5}
The $\xi$-{\it Artin--Schelter regularity} (or {\it $\xi$-AS 
regularity}) of a connected graded algebra $A$ is
$$\ASregxi(A):=\Torregxi(\Bbbk)+ \CMregxi(A).$$
\end{definition}

If there exists a noetherian AS regular algebra $T$ and a 
finite map $T \to A$, then it follows from 
Proposition~\ref{xxpro5.8} that there is a $\xi$ such that 
$\ASregxi(A)<\infty$. 

\begin{remark}
\label{xxrem2.6}
\begin{enumerate}
\item[(1)]
When $\xi_0 > 0$, it follows from Theorem \ref{xxthm3.3} below 
and Example~\ref{xxex2.4}(2) that $\Torregxi(_A \Bbbk)\geq 
-\CMregxi(A)$, or equivalently $\ASregxi(A)\geq 0$. Let $T$ 
be a non-Koszul AS regular algebra of global dimension 3
that is generated in degree 1. Then $T$ is of type $(3,4)$ 
and 
$$t^T_i(\Bbbk)=\begin{cases} 0, & i=0,\\
1, & i=1,\\
3, & i=2,\\
4, & i=3.
\end{cases}
$$
By Example \ref{xxex2.2}(3), $\CMregxi(T)=-4 \xi_0+ 3 \xi_1$ 
and it is easy to check that, if $\xi_0=1$, then
$$\Torregxi(\Bbbk)=\max\{0, 1-\xi_1, 3-2\xi_1, 4-3\xi_1\}
=\begin{cases} 4-3\xi_1 & \xi_1\leq 1,\\
3-2\xi_1, & 1\leq \xi_1\leq 1.5,\\
0, & 1.5\leq \xi_1.
\end{cases}
$$
As a consequence, $\Torregxi(_T \Bbbk)=-\CMregxi(T)$ if 
$\xi_1\leq 1$ and $\Torregxi(_T \Bbbk)> -\CMregxi(T)$ 
if $\xi_1>1$. 
Note that for any $\xi$, by Example \ref{xxex2.2}(1),
$$\CMregxi(\Bbbk)=0.$$
So, if $\xi_1>1$, Theorem \ref{xxthm0.4}
(or Theorem \ref{xxthm3.10}) fails even when $X=\Bbbk$.
\item[(2)]
Let $T$ be as in part (1) and let $B=T[x_1,
\cdots,x_n]$ for some integer $n\geq 1$, where $\deg x_s=1$
for all $1\leq s\leq n$. 
Then $B$ is of type $(3+n,4+n)$ 
and 
$$t^B_i(_B\Bbbk)=\begin{cases} 0, & i=0,\\
1, & i=1,\\
i+1, & 2\leq i\leq n+3.
\end{cases}
$$
By Example \ref{xxex2.2}(4), 
$\CMregxi(B)=-(n+4) \xi_0+ (n+3)\xi_1$ 
and it is easy to check that, if $\xi_0=1$, then
$$\Torregxi(\Bbbk_B)
=\begin{cases} (n+4)-(n+3)\xi_1 & \xi_1\leq 1,\\
3-2\xi_1, & 1\leq \xi_1\leq 1.5,\\
0, & 1.5\leq \xi_1.
\end{cases}
$$
Similarly, $\Torregxi(\Bbbk_B)> -\CMregxi(B)$ if $\xi_1>1$.
\item[(3)]
Suppose $\xi_0=1$ and $\xi_1\leq 1$. Let $T$ be a noetherian 
AS regular algebra. By \cite[(3--4), p.1600]{StZ},
$t^T_{i-1}(\Bbbk)<t^T_{i}(\Bbbk)$ for all $1\leq i \leq 
d:=\gldim T$. Since each $t^T_i(\Bbbk)$ is an integer, we 
have $t^T_{i-1}(\Bbbk)-(i-1)\leq t^T_{i}(\Bbbk)-i$. Since 
$\xi_1\leq 1$, we obtain that
$$t^T_{i-1}(\Bbbk)-\xi_1 (i-1)\leq t^T_{i}(\Bbbk)-\xi_1 i$$
for all $1\leq i\leq d$. This implies that
$$\Torregxi(\Bbbk)=t^T_{d}(\Bbbk)-\xi_1 d=\bfl -\xi_1 d$$
by \cite[Proposition 3.1(4)]{StZ}. By Example \ref{xxex2.2}(3), 
$$\Torregxi(_T\Bbbk)=-\CMregxi(T),$$ 
or equivalently
$$\ASregxi(T)=0$$
(compare to Theorem \ref{xxthm0.6}).
\item[(4)]
Let $\xi=(1,\xi_1)$. Let $A=\Bbbk[x]$ with $\deg (x)=2$. Then 
$A$ is of type $(1,2)$ and $\deg \Tor^1_A(\Bbbk,\Bbbk)=2$. We 
have 
$$\CMregxi(A)=-2+\xi_1$$
and
$$\Torregxi(\Bbbk)=\begin{cases} 
2-\xi_1 & \xi_1\leq 2,\\
0 & \xi_1>2.
\end{cases}$$
As a consequence,
$$\ASregxi(\Bbbk)=\begin{cases} 
0 & \xi_1\leq 2,\\
-2+\xi_1 & \xi_1>2.
\end{cases}$$
\end{enumerate}
\end{remark}

The following is a generalization of \cite[Lemma 2.3]{KWZ2}. 

\begin{lemma}
\label{xxlem2.7}
Assume that $\xi_0>0$. Suppose that $T$ and $A$ are connected 
graded algebras. Then
\[ \Torregxi({}_{T \otimes A} \Bbbk) =
\Torregxi({}_{T} \Bbbk) + \Torregxi({}_{A} \Bbbk).\]
\end{lemma}

\begin{proof}
Let $P$ be a projective resolution of $\Bbbk$ as a right 
$T$-module and let $Q$ be a projective resolution of $\Bbbk$ as 
a right $A$-module. Let $X$ and $Y$ denote the complexes given 
by tensoring $P$ and $Q$ with ${}_T \Bbbk$ and ${}_A \Bbbk$ 
respectively. Then $\Tor_i^T(\Bbbk, \Bbbk)$ and 
$\Tor_i^A(\Bbbk, \Bbbk)$ can be computed by taking homology of 
$X$ and $Y$, respectively. Further, 
$\Tor_i^{T \otimes A}(\Bbbk, \Bbbk)$ can be computed by taking 
homology of the complex $X \otimes Y$. By the K\"{u}nneth formula 
(see, e.g. \cite[Theorem 10.8.1]{Ro}), we have
$$\begin{aligned}
\bigoplus_{p+q = n} \Tor_p^T(\Bbbk,\Bbbk) \otimes 
\Tor_q^A(\Bbbk,\Bbbk) &\cong 
\bigoplus_{p+q = n} H_p(X) \otimes H_q(Y) \\
&\cong H_n(X \otimes Y) \cong \Tor_n^{T \otimes A}(\Bbbk,\Bbbk)
\end{aligned}
$$
where $\otimes$ stands for $\otimes_{\Bbbk}$. Therefore,
$$\begin{aligned}
\Torregxi ({}_{T \otimes A} \Bbbk )
&=\sup_{n \in \mathbb{Z}}\{ \xi_0 \deg (\Tor^{T \otimes A}_n
   (\Bbbk, \Bbbk)) -\xi_1 n \} \\
&=\sup_{p, q \in \mathbb{Z}}\{ \xi_0 \deg (\Tor^{T}_{p}
  (\Bbbk, \Bbbk)) + \xi_0 \deg(\Tor^{A}_{q}(\Bbbk, \Bbbk)) 
	-\xi_1 (p+q) \} \\
&=\sup_{p \in \mathbb{Z}}\{ \xi_0 \deg (\Tor^{T}_{p}
  (\Bbbk, \Bbbk)) -\xi_1 p \} + \sup_{ q \in \mathbb{Z}}
	\{  \xi_0 \deg(\Tor^{A}_{q}(\Bbbk, \Bbbk)) -\xi_1 q\} \\
&= \Torregxi({}_T\Bbbk) + \Torregxi({}_A \Bbbk),
\end{aligned}$$
as desired.
\end{proof}

\begin{example}
\label{xxex2.8}
Let $\xi_0=1$.

Let $T$ be the noetherian AS regular algebra given in Remark 
\ref{xxrem2.6}(1). Then 
$$\Torregxi(_T \Bbbk)=\max\{0, 1-\xi_1, 3-2\xi_1, 4-3\xi_1\}
=\begin{cases} 4-3\xi_1 & \xi_1\leq 1,\\
3-2\xi_1, & 1\leq \xi_1\leq 1.5,\\
0, & 1.5\leq \xi_1.
\end{cases}
$$

Let $A$ be an affine (commutative) noetherian Koszul algebra that has 
infinite global dimension. Then 
$$\Torregxi(_A \Bbbk)=\max_{n\in {\mathbb N}}\{n-n\xi\}
=\begin{cases} \infty & \xi_1>1,\\
0& \xi_1 \leq 1.
\end{cases}
$$

By Lemma \ref{xxlem2.7}, we obtain that
$$\Torregxi(_{A\otimes T} \Bbbk)
=\begin{cases} \infty & \xi_1>1,\\
4-3\xi_1 & \xi_1\leq 1.
\end{cases}
$$
\end{example}

\section{Equalities and inequalities}
\label{xxsec3}

In this section we study the relationships between the 
weighted regularities defined in the previous section, 
generalizing results of J{\o}rgensen, Dong, and Wu
\cite{Jo3, Jo4, DW} and proving Theorem \ref{xxthm0.3}.
In this section we 
fix a pair of real numbers $\xi = (\xi_0, \xi_1)$.

Recall that for a cochain complex $X$ and $\ell \in \mathbb{Z}$, 
the complex $X(\ell)$ shifts the degrees of each graded vector 
space $X(\ell)^i_m = X^i_{m+\ell}$ while the complex $X[\ell]$ 
shifts the complex $X[\ell]^i = X^{i+\ell}$. 

\begin{lemma}
\label{xxlem3.1} 
Let $A$ be a noetherian connected graded algebra and $X$ be a 
nonzero complex of graded left $A$-modules. The following 
statements hold.
\begin{enumerate}
\item[(1)]
Let $\lambda>0$ and $\xi'=\lambda \xi$. Then 
$$\deg_{\xi'}(X)=\lambda \degxi(X).$$
The above equation also holds if $\deg$ is replaced with 
$\ged$, $\CMreg$, $\Extreg$, or $\Torreg$.
\item[(2)]
Suppose $\degxi(X)$ is finite. Then 
$$\degxi(X[1])=\degxi(X)-\xi_1 \quad {\text{and}}
\quad \degxi(X(1))=\degxi(X)-\xi_0.$$ 
Similar equations hold if $\deg$ is replaced with $\ged$, $\CMreg$, $\Extreg$, or $\Torreg$.
\item[(3)]
Suppose $\degxi(X)$ is finite. Assume that $\xi_0$ and $\xi_1$ 
are nonzero such that $\xi_0^{-1}\xi_1$ is irrational. Then the
numbers in the collection 
$$\left\{\degxi(X[m](n))\right\}_{m, n \in \mathbb{Z}}$$ 
are distinct. This assertion holds if $\deg$ is replaced with $\ged$, $\CMreg$, $\Extreg$, or $\Torreg$.
\end{enumerate}
In the following parts, we assume that 
$\{\xi_n:=(\xi_{0,n},\xi_{1,n})\}_{n\geq 1}$ 
is a sequence of 
pairs such that $\lim\limits_{n\to\infty} \xi_n=\xi$. In parts 
{\rm{(4)}} and {\rm{(5)}}, we further assume that $\xi_0>0$.
\begin{enumerate}
\item[(4)]
Suppose $X \in \D^{\bb}(A \Gr)$. If $\deg H^j(X) < \infty$ 
for all $j$, then 
$$\lim\limits_{n\to\infty} \deg_{\xi_{n}}(X)=\degxi(X).$$ 
Similarly, if $\ged H^j(X) > -\infty$ for all $j$,
then 
$$\lim\limits_{n\to\infty} \ged_{\xi_{n}}(X)=\gedxi(X).$$
\item[(5)] 
If $Y$ is a nonzero object in $\D^{\bb}_{\fg}(A \Gr)$, then
$\lim\limits_{n\to\infty} \CMreg_{\xi_{n}}(Y)=\CMregxi(Y)$.
\item[(6)] 
If $Y$ is a nonzero object in $\D^{\bb}_{\fg}(A \Gr)$ of finite 
projective dimension, then 
$\lim\limits_{n\to\infty} \Torreg_{\xi_{n}}(Y)=\Torregxi(Y)$ and
$\lim\limits_{n\to\infty} \Extreg_{\xi_{n}}(Y)=\Extregxi(Y)$.
\end{enumerate}
\end{lemma}

\begin{proof}
(1) Let $\lambda > 0$ and $\xi' = \lambda \xi$. If $\degxi(X)$ 
is finite, then 
\begin{align*} 
\deg_{\xi '} (X) &= \sup_{m,n \in \mathbb{Z}} 
\{ \lambda \xi_0 m+ \lambda \xi_1 n \mid H^n(X)_m\neq 0\} \\
&=\lambda \sup_{m,n \in \mathbb{Z}} 
\{ \xi_0 m+ \xi_1 n \mid H^n(X)_m\neq 0\} 
=\lambda \degxi(X).
\end{align*}
If $\degxi(X)$ is infinite, then so is $\deg_{\xi'}(X)$. 
A similar proof works for $\gedxi(X)$, and the result holds 
for $\CMregxi(X)$, $\Extregxi(X)$, and $\Torregxi(X)$, since 
they can each be expressed as $\degxi$ or $\gedxi$ of certain 
complexes of $A$-modules.

(2) Suppose that $\degxi(X)$ is finite. Then
\begin{align*}
\degxi(X[1]) &= \sup_{m,n \in \mathbb{Z}} 
\{ \xi_0 m+ \xi_1 n \mid    H^n(X[1])_m\neq 0\} \\ 
&= \sup_{m,n \in \mathbb{Z}} 
\{ \xi_0 m+ \xi_1 n \mid    H^{n+1}(X)_m\neq 0\} \\
&= \sup_{m,n \in \mathbb{Z}} 
\{ \xi_0 m+ \xi_1 (n - 1) \mid    H^{n}(X)_m\neq 0\} 
= \degxi(X) - \xi_1.
\end{align*}
Further,
\begin{align*}
\degxi(X(1)) &= \sup_{m,n \in \mathbb{Z}} 
\{ \xi_0 m+ \xi_1 n \mid    H^n(X(1))_m\neq 0\} \\ 
&= \sup_{m,n \in \mathbb{Z}} 
\{ \xi_0 m+ \xi_1 n \mid    H^{n}(X)_{m+1}\neq 0\} \\
&= \sup_{m,n \in \mathbb{Z}} 
\{ \xi_0 (m - 1)+ \xi_1 n \mid    H^{n}(X)_m\neq 0\} 
= \degxi(X) - \xi_0.
\end{align*}
The proofs for $\gedxi(X)$, $\CMregxi(X)$, $\Torregxi(X)$, and 
$\Extregxi(X)$ are similar.

(3) Suppose $\degxi(X) = d$ is finite and that $\xi_0, \xi_1$ 
are nonzero such that $\xi_0^{-1}\xi_1$ is irrational. By the 
above result, for $m, n \in \mathbb{Z}$, we have 
$\degxi(X[m](n)) = d - m \xi_1 - n \xi_0$. 
If $\degxi(X[m](n)) = \degxi(X[m'](n'))$ then, 
$(m - m')\xi_1 = (n' - n)\xi_0$. Since $\xi_0^{-1} \xi_1$ is 
irrational and $\xi_0,\xi_1 \neq 0$, therefore $m = m'$ and 
$n = n'$. The same proof shows that the result holds for 
$\gedxi(X)$, $\CMregxi(X)$, $\Torregxi(X)$, and $\Extregxi(X)$.

(4) Let $\{\xi_n = (\xi_{0,n}, \xi_{1,n})\}_{n \geq 1}$ be a 
sequence such that $\lim\limits_{n \to \infty} \xi_n = \xi$. Then
\begin{align*}
\lim_{n \to \infty} \deg_{\xi_n}(X) 
= \lim_{n \to \infty} \sup_{i,j \in \mathbb{Z}} 
\{  \xi_{0,n} i+  \xi_{1,n} j \mid H^j(X)_i\neq 0\}
\end{align*}
Since $X \in \D^{\bb}(A \Gr)$, only finitely many $H^j(X)$ are 
nonzero and so the above supremum can be taken over finitely many $j$. 

If $\deg H^j(X) < \infty$ for each $j$, by hypothesis $\xi_0>0$, 
then the supremum is taken over a finite set, and so the convergence 
holds. If some $\deg H^j(X) = \infty$ (and $\xi_0 > 0$), then 
$\degxi(X)$ is infinite, and $\deg_{\xi_n}(X)$ is also infinite when 
$n\gg 0$. A similar proof shows that the assertion for $\gedxi(X)$ holds.

(5) If $Y \in \D^{\bb}_{\fg}(A \Gr)$ is nonzero, then 
$X:=\R \Gamma_{\fm}(Y)$ is a nonzero object in $\D^{\bb}(A \Gr)$. 
By Theorem \ref{xxthm1.4}(2), $\deg H^i(X)<\infty$ for all $i$.
Hence, by part (4), the assertion holds for 
$\CMregxi(Y) = \degxi(\R \Gamma_{\fm}(Y))=\degxi(X)$.

(6) If $Y \in \D^{\bb}_{\fg}(A \Gr)$ is nonzero with finite 
projective dimension, then $\Bbbk \otimes_A^L Y$ and 
$\R \Hom_A(Y, \Bbbk)$ are nonzero in $\D^{\bb}(\Bbbk \Gr)$.
Further, in each degree of the complex, homology is 
finite-dimensional. Hence, by part (4), the assertion holds 
for $\Torregxi(Y) = \degxi(\Bbbk \otimes_A^L Y)$ and 
$\Extregxi(Y) = \gedxi \R \Hom_A(Y, \Bbbk)$.
\end{proof}

\begin{proposition}
\label{xxpro3.2} 
Suppose that $A$ is a noetherian connected graded algebra 
with balanced dualizing complex. Assume that $\xi_0>0$. Let 
$Z$ be a nonzero object in $\D^{-}(A\Gr)$ with 
$\deg_{\xi}(Z)$ finite. Then
$$\Extregxi(Z) \leq \degxi(Z) +\Torregxi(\Bbbk).$$
\end{proposition}

\begin{proof} By Lemma \ref{xxlem3.1}(1), we may assume that 
$\xi_0=1$. Let $p:= \degxi(Z)$, which is finite by hypothesis. 
If $\Torregxi(\Bbbk)=\infty$, then the assertion holds trivially, 
so we may assume that $r:= \Torregxi(\Bbbk)$ is finite.

By Example \ref{xxex2.4}(1) and our assumption that $\xi_0=1$, 
\begin{equation}
\label{E3.2.1}\tag{E3.2.1}
\deg (\Tor^A_{i}(\Bbbk, \Bbbk))\leq r+\xi_1 i
\end{equation}
for all $i$. 

Let 
\begin{equation}
\label{E3.2.2}\tag{E3.2.2}
F: \qquad \cdots \to F_{s} \to F_{s-1} \to \cdots \to F_1 
\to A \to \Bbbk\to 0
\end{equation}
be a minimal projective resolution of the right $A$-module 
$\Bbbk$. By \eqref{E3.2.1}, the generators of $F_m$ are placed 
in degrees less than or equal to $r+ \xi_1 m$ for every $m$ and 
so $F_m$ can be written as a finite coproduct 
$$F_m=\coprod_{j} A (-\sigma_{m,j})$$
where $\sigma_{m,j}$ are integers $\leq r+\xi_1 m$. Taking 
Matlis duals in (\ref{E3.2.2}), $I:=F'$ is a minimal injective 
resolution of the left $A$-module $\Bbbk$ which has
$$I^m=\coprod_{j} A' (\sigma_{m,j})$$
for each $m\geq 0$.

Since $\xi_0 = 1$ and $p=\degxi(Z)$ (or equivalently 
$\deg H^n(Z) \leq p-\xi_1 n$ for all $n$), we have 
$$H^{-n}(Z)_{>p+\xi_1 n} =0$$
for all $n$. For each $m$, $\Ext_A^m(H^{-n}(Z), \Bbbk)$ is a 
subquotient of $\Hom_A(H^{-n}(Z),I^{m})$, which is 
$$
\Hom_A \left(H^{-n}(Z), \coprod_{j} A'(\sigma_{m,j})\right)
\cong \coprod_{j} (H^{-n}(Z))' (\sigma_{m,j}),
$$
and this vanishes in degree less than $-p-\xi_1 n-r-\xi_1 m$.
Hence
\begin{equation}
\label{E3.2.3}\tag{E3.2.3}
\Ext^m_A(H^{-n}(Z),\Bbbk)_{<-p -\xi_1 n -r - \xi_1 m}=0
\end{equation}
for all $m,n$.  By \cite[Lemma 1.2]{Jo4}, there is a convergent 
spectral sequence
$$E^{m,n}_2:=\Ext^{m}_A(H^{-n}(Z), \Bbbk)\Longrightarrow 
\Ext^{m+n}_A(Z,\Bbbk),$$
and since \eqref{E3.2.3} shows that 
$(E^{m,n}_2)_{<-p -r-\xi_1 (m+n)}=0$, it follows that
$$\Ext^q_A(Z,\Bbbk)_{<-p-r-\xi_1 q}=0.$$
This condition is equivalent to 
$\ged(\Ext^q_A(Z,\Bbbk))\geq -p-r-\xi_1 q$ for all $q$. By 
Definition \ref{xxdef2.3}, $\Extregxi(Z)\leq p+r$ as desired.
\end{proof}

The following is a generalization of \cite[Theorem 2.5]{Jo4}, and 
establishes Theorem \ref{xxthm0.3}(1). Combining this result with 
Example \ref{xxex2.4}(2) also yields Theorem \ref{xxthm0.3}(3).
 
\begin{theorem} 
\label{xxthm3.3}
Assume that $\xi_0>0$. Let $A$ be a noetherian connected graded 
algebra with a balanced dualizing complex and let $X$ be a 
nonzero object in $\D^{\bb}_{\fg}(A\Gr)$. Then
$$\Torregxi(X) = \Extregxi(X)\leq \CMregxi(X) + \Torregxi(\Bbbk).$$
\end{theorem}

\begin{proof}
Let $Z=\R\Gamma_{\fm}(X)$. By Theorem \ref{xxthm1.4}(2), $Z'$ is 
nonzero in $\D^{\bb}_{\fg}(A^{\op}\Gr)$. Hence 
$\degxi(Z)$ is a finite number and $Z\in \D^{\bb}(A\Gr)$. By 
Proposition \ref{xxpro3.2},
$$\Extregxi(Z) \leq \degxi(Z) +\Torregxi(\Bbbk).$$
By definition,
$$\CMregxi(X)=\degxi(\R\Gamma_{\fm}(X))=\degxi(Z).$$
By \cite[Proposition 1.1]{Jo4},
$$\RHom_A(X,\Bbbk)\cong \RHom_{A}(\R\Gamma_{\fm}(X),\Bbbk)
= \RHom_{A}(Z, \Bbbk).$$ 
Hence $\Extregxi(X)=\Extregxi(Z)$ and so the assertion follows.
\end{proof}

Note that \cite[Theorem 2.5]{Jo4} is a special case of the above 
theorem where $\xi=(1,1)$. By considering some specific weights, 
we obtain the following corollary.

\begin{corollary}
\label{xxcor3.4}
Assume that $\xi_0 > 0$. Let $A$ be a noetherian connected graded 
algebra with a balanced dualizing complex and let $X$ be a nonzero 
object in $\D^{\bb}_{\fg}(A\Gr)$. If $\Torregxi(\Bbbk)$ is finite, 
then so is $\Torregxi(X)$.
\end{corollary}

\begin{proof}
Since $X$ has a finitely generated minimal free resolution, 
$\Extregxi(X) = \Torregxi(X)$. Hence, by Theorem~\ref{xxthm3.3}, 
\[\Torregxi(X) \leq \CMregxi(X) + \Torregxi(\Bbbk),
\] 
and since $\CMregxi(X)$ is finite (see comments after Definition
\ref{xxdef2.1}), the result follows.
\end{proof}

\begin{proof}[Proof of Theorem \ref{xxthm0.3}(3)]
Let $X=A$ in Theorem \ref{xxthm3.3}, we obtain that
\[
\ASregxi(A):=\CMregxi(A) + \Torregxi(\Bbbk)
\geq \Extregxi(A) =0. %\qedhere
\]
\end{proof}

Theorem \ref{xxthm3.3} has other consequences by setting $\xi$ 
to be special values. For example, if $A$ is AS regular, when 
we set $\xi_1 = 1$ and take the limit as $\xi_0\to 0^{+}$, we
obtain that 
\begin{equation}
\label{E3.4.1}\tag{E3.4.1}
\sup(X)\leq \lcd(X).
\end{equation}

The next theorem is a generalization of \cite[Theorem 2.6]{Jo4}, 
and it  provides a proof of Theorem \ref{xxthm0.3}(2). 

\begin{theorem} 
\label{xxthm3.5}
Assume that $\xi_0>0$. Let $A$ be a noetherian connected graded 
algebra with a balanced dualizing complex and let $X$ be a 
nonzero object in $\D^{\bb}_{\fg}(A\Gr)$. Then
$$\CMregxi(X)\leq \Extregxi(X)+ \CMregxi(A).$$
\end{theorem}

\begin{proof} By Lemma \ref{xxlem3.1}(1), we may assume that
$\xi_0=1$. As noted after Definition~\ref{xxdef2.1}, 
$\CMregxi(A)$ is finite. If $\Extregxi(X)$ is infinite, then 
the assertion holds trivially. So we may assume that
$r:=\Extregxi(X)$ is finite, and hence $\Torregxi(X)=r$
is finite. Let $F$ be a minimal free resolution of $X$. By 
definition,
$$\deg (\Tor^A_i(\Bbbk, X))\leq r+\xi_1 i$$
for all integers $i$. This implies that the generators of $F_m$ 
are in degrees less than or equal to $r+\xi_1 m$ for 
each integer $m$. Therefore $F_m$ can be written as a finite 
coproduct
$$F_m=\coprod_{j} A(-\sigma_{ m,j})$$
where $\sigma_{m,j}$ are integers $\leq r+\xi_1 m$.

Let $p := \CMregxi(A)$. By \cite[Observation 2.3]{Jo4}, 
$\CMregxi(A_A)=\CMregxi(_AA)$, so we have
\begin{equation}
\label{E3.5.1}\tag{E3.5.1}
H^n_{\fm^{op}}(A)_{>p-\xi_1 n}=0
\end{equation}
for each integer $n$. 

Now $\Tor^A_{-m}(H^n_{\fm^{op}}(A), X)$ is a subquotient 
of $H^n_{\fm^{op}}(A)\otimes_A F_{-m}$ which is 
$$H^n_{\fm^{op}}(A)\otimes_A \coprod_{j} 
A(-\sigma_{-m,j}) \cong 
\coprod_{j} H^n_{\fm^{op}}(A)(-\sigma_{ -m,j}),$$
and the latter vanishes in degree larger than 
$p-\xi_1 n + r -\xi_1 m$ by \eqref{E3.5.1}. Hence 
\begin{equation}
\label{E3.5.2}\tag{E3.5.2}
\Tor^A_{-m}(H^n_{\fm^{\op}}(A),X)_{>p+r -\xi_1 n-\xi_1 m}=0
\end{equation}
for all $m,n$. 

By \cite[Lemma 1.3]{Jo4}, there is a convergent spectral 
sequence
$$E^{m,n}_2:=\Tor^A_{-m}(H^n_{\fm^{op}}(A),X) 
\Longrightarrow H^{m+n}_{\fm}(X).$$
Since \eqref{E3.5.2} shows that
$(E^{m,n}_2)_{>p+r -\xi_1 (n+m)}=0$ for all $m,n$, the 
spectral sequence implies that
$$H^q_{\fm}(X)_{>p+r-\xi_1 q}=0$$
for all $q$. By definition, this is equivalent to 
$$\CMregxi(X)\leq r + p=\Extregxi(X)+\CMregxi(A). 
$$
\end{proof}

If $\xi=(1,1)$, then the above theorem recovers 
\cite[Theorem 2.6]{Jo4}. The following is an immediate 
corollary of Theorems \ref{xxthm3.3} and \ref{xxthm3.5}.

\begin{corollary}
\label{xxcor3.6}
Assume that $\xi_0>0$. Let $A$ be a noetherian connected 
graded algebra with a balanced dualizing complex. Suppose 
that $\Torregxi(\Bbbk)=-\CMregxi(A)$ {\rm{(}}which is a 
finite number, denoted by $c${\rm{)}}. Then for all nonzero 
$X$ in $\D^{\bb}_{\fg}(A \Gr)$, $\Extregxi(X)=\CMregxi(X)+c$.
\end{corollary}

In the next lemma, $\reg_{\xi}$ can be $\CMregxi$, or 
$\Extregxi$ or $\Torregxi$. 

\begin{lemma}
\label{xxlem3.7}
Suppose that $\xi = (\xi_0, \xi_1)$ with $\xi_0>0$.
Let $X\to Y \to Z\to X[1]$ be a distinguished triangle 
in $\D(A\Gr)$. Then
\begin{enumerate}
\item[(1)]
$\reg_{\xi}(Y) \leq \max\{ \reg_{\xi}(X), \reg_{\xi}(Z)\}$. 
\item[(2)]
$\reg_{\xi}(X) \leq \max \{\reg_{\xi}(Y), \reg_{\xi}(Z) + \xi_1\}$.
\item[(3)]
$\reg_{\xi}(Z) \leq \max\{ \reg_{\xi}(X)-\xi_1, \reg_{\xi}(Y)\}$.
\end{enumerate}
\end{lemma}

\begin{proof}
We will prove the assertions only for $\reg_{\xi}=\CMregxi$.
The other proofs are similar.

(1) Starting from the distinguished  triangle $X\to Y \to Z\to X[1]$, 
we obtain a distinguished triangle 
$$\R\Gamma_{\fm}(X)\to \R\Gamma_{\fm}(Y) \to \R\Gamma_{\fm}(Z)
\to \R\Gamma_{\fm}(X)[1]$$
which implies that there is a long exact sequence
$$\cdots \to H^{n-1}_{\fm}(Z)\to H^{n}_{\fm}(X)\to H^{n}_{\fm}(Y)
\to H^{n}_{\fm}(Z)\to H^{n+1}_{\fm}(X)\to \cdots.$$
Hence, for all integers $n$, we have $\deg H^{n}_{\fm}(Y)\leq 
\max\{ \deg H^{n}_{\fm}(X), \deg H^{n}_{\fm}(Z)\}$.
Since $\xi_0>0$, we obtain that 
$$\begin{aligned}
\xi_0 \deg H^{n}_{\fm}(Y)+ \xi_1 n
&\leq 
\max\{\xi_0 \deg H^{n}_{\fm}(X) +\xi_1 n,\xi_0 \deg H^{n}_{\fm}(Z)+\xi_1 n\}\\
&\leq \max\{\CMregxi(X), \CMregxi(Z)\}
\end{aligned} 
$$ 
for all $n$. This implies that 
$$\CMregxi(Y)\leq \max\{\CMregxi(X), \CMregxi(Z)\}$$
as desired. 

(2) By rotation, we have a distinguished  triangle
$Z[-1]\to X\to Y \to Z$. By part (1), we have 
$$\CMregxi(X)\leq \max\{ \CMregxi(Y), \CMregxi(Z[-1])\}.$$
The assertion now follows from Lemma \ref{xxlem3.1}(2).

(3) The proof is similar to the proof of part (2).
\end{proof}

Our next result is a generalization \cite[Proposition 5.6]{DW}.
We will use the following notation.
Let 
\begin{align}
\label{E3.7.1}\tag{E3.7.1}
Y:
&=\cdots \to 0\to F^{-w}\to F^{-(w-1)}\to \cdots \to F^{-1}
\to F^0\to 0\to \cdots\\
&=\cdots \to 0\to F_{w}\; \to \; F_{w-1}\; \to \; \cdots \; 
\to \;\; F_1 \; \to \; F_0 \; \to \; 0\to \cdots\notag
\end{align}
be a minimal free resolution of a complex in 
$\D^{\bb}_{\fg}(A\Gr)$ of finite projective dimension that is 
bounded below at position 0. Let 
$$Z=F_w[w],
\quad {\text{and}}\quad 
X=\cdots \to 0\to 0\to F_{w-1}\to \cdots \to F_1\to 
  F_0\to 0\to \cdots.$$
Observe that we have a distinguished triangle
\begin{equation}
\label{E3.7.2}\tag{E3.7.2}
X\to Y\to Z\to X[1].
\end{equation}
For each $0 \leq s \leq w,$ write 
\begin{equation}
\label{E3.7.3}\tag{E3.7.3}
F^{-s}=F_s:=\coprod_{j=1}^{n_s} A(-\sigma_{s,j})
\end{equation}
for some integers $\sigma_{s,j}$ and write 
\begin{equation}
\label{E3.7.4}\tag{E3.7.4}
\sigma_s:=\max_{1 \leq j \leq n_s}\{\sigma_{s,j}\}.
\end{equation}

\begin{lemma}
\label{xxlem3.8}
Retain the above notation {\rm{(\eqref{E3.7.1} - \eqref{E3.7.4})}}. 
Let $\xi_0>0$.
\begin{enumerate}
\item[(1)]
Then 
$$\Torregxi(Y)=\max_{0\leq s \leq w}\{\Torregxi(F_{s}[s])\}.$$
As a consequence,
$$\Torregxi(Y)=\max\{ \Torregxi(X),\Torregxi(Z)\}.$$
\item[(2)]
If $Z\neq 0$ and $\xi_1\ll 0$, then $\Torregxi(Y)=\Torregxi(Z)$.
\end{enumerate}
\end{lemma}

\begin{proof} (1) For each $0 \leq s \leq w$, write 
$F_s=\coprod_{j = 1}^{n_s} A(-\sigma_{s,j})$ as in \eqref{E3.7.3}. 
By definition,
$\Tor^A_s(\Bbbk, Y)=\coprod_{j=1}^{n_s} \Bbbk (-\sigma_{s,j})$
and consequently,
\begin{equation}
\label{E3.8.1}\tag{E3.8.1}
\Torregxi(Y)
= \max_{0 \leq s \leq w}\left\{ \xi_0 \max_{1 \leq j \leq n_s} 
  \{\sigma_{s,j}\} -\xi_1 s\right\}
=\max_{0 \leq s \leq w}\left\{ \xi_0\sigma_{s} -\xi_1 s\right\}.
\end{equation}
Since $\Torregxi(F_s[s]) = \xi_0 \sigma_{s} - \xi_1 s$, the main 
assertion follows, and the consequence follows from the main 
assertion.

(2) Since $Z\neq 0$, $\sigma_w:=\max_{1 \leq j \leq n_w} 
\{\sigma_{w,j}\}$ is finite. It follows from \eqref{E3.8.1} that, 
when $\xi_1\ll 0$, $\Torregxi(Y)=\sigma_w -\xi_1 w$. Since 
$\Torregxi(Z)=\sigma_w -\xi_1 w$ by definition, the assertion 
follows.
\end{proof}

We continue to use the notation introduced before Lemma 
\ref{xxlem3.8}.

\begin{lemma}
\label{xxlem3.9}
Retain the above notation. Suppose that $Z\neq 0$ and that 
$\xi_0>0$.
\begin{enumerate}
\item[(1)]
$$\CMregxi(Y)\leq \max\limits_{0\leq s \leq w}
\{\CMregxi(F_s[s])\}.$$
\item[(2)]
$$\depth(Y)\geq \min_{1\leq s\leq w} 
\{\depth(F_s[s])\}=-w+\depth(A)=\depth(Z).$$
\item[(3)]
Let $d=\depth(A)$ and $f=-w+\depth(A)$. Then 
$$\deg H^{f}_{\fm}(Y)\leq 
\deg H^{d}_{\fm}(A)+\sigma_{w}
=\deg H^f_{\fm}(F_w[w])=\deg H^f_{\fm}(Z).$$
\item[(4)]
Let $\alpha=\deg H^f_{\fm}(Z)
(=\deg H^{d}_{\fm}(A)+\sigma_{w})$. Then the natural map 
$$H^f_{\fm}(Z)_{\alpha}\to H^{f+1}_{\fm}(X)_{\alpha}
=H_{\fm}^{f}(X[1])_{\alpha}$$
is zero.
\item[(5)]
If $\xi_1\ll 0$, then 
$$\CMregxi(Y)=\max\limits_{0\leq s \leq w}
\{\CMregxi(F_s[s])\}=\CMregxi(F_w[w])=\CMregxi(Z).$$
\end{enumerate}
\end{lemma}

\begin{proof} 
(1) This follows by induction on $w$, \eqref{E3.7.2}, and  
Lemma \ref{xxlem3.7}(1).

(2) First of all it is clear that 
$$\depth(Z)=\depth(F_w[w])=
\depth\left(\coprod_{j=1}^{n_w} A(-\sigma_{w,j})[w]\right)
=-w+\depth(A).$$
It is obvious that $\min_{1\leq s\leq w} 
\{\depth(F_s[s])\}=-w+\depth(A)$.

Next, we prove $\depth(Y)\geq \min_{1\leq s\leq w} 
\{\depth(F_s[s])\}$ by induction on $w$. 
Let ${\overline{Y}}=\R\Gamma_{\fm}(Y)$, ${\overline{X}}=\R\Gamma_{\fm}(X)$, and 
${\overline{Z}}=\R\Gamma_{\fm}(Z)$.  By Example \ref{xxex2.2}(6), 
$\depth(Y)=\inf({\overline{Y}})$. A similar assertion holds for $X$ and $Z$. 
It follows from \eqref{E3.7.2} that there is a distinguished 
triangle
\begin{equation}
\label{E3.9.1}\tag{E3.9.1}
{\overline{X}}\to {\overline{Y}}\to {\overline{Z}}\to {\overline{X}}[1].
\end{equation}
Hence $\inf({\overline{Y}})\geq \min\{\inf({\overline{X}}),
\inf({\overline{Z}})\}$, consequently,
$$\begin{aligned}
\depth(Y)&=\inf({\overline{Y}})\geq \min\{\inf({\overline{X}}),
\inf({\overline{Z}})\}
=\min\{\depth(X),\depth(Z)\}.
\end{aligned}$$ 
By definition, $X$ is a version of $Y$ with $w$ being replaced
by $w-1$. Hence the first inequality follows by the above 
inequality and induction on $w$.

(3) It is clear that 
$$\deg H^f_{\fm}(Z)=\deg H^f_{\fm}(F_w[w])=
\deg H^{d}_{\fm}(A)+\sigma_{w}.$$ It remains to 
show that $\deg H_{\fm}^{f}(Y)\leq \deg H_{\fm}^{f}(Z)$. 

By part (2), 
$$\depth(X)\geq \min_{0\leq s\leq w-1} 
\{\depth(F_s[s])\}\geq -(w-1)+\depth(A)=f+1.$$ 
As a consequence, $H_{\fm}^{i}(X)=0$ for all $i\leq f$. 
Taking the cohomology group of \eqref{E3.9.1} gives rise to
a long exact sequence
\begin{equation}
\label{E3.9.2}\tag{E3.9.2}
\cdots \to 0\to 0(=H_{\fm}^{f}(X)) \to H_{\fm}^{f}(Y)\to 
H_{\fm}^{f}(Z) \to H_{\fm}^{f+1}(X)\to \cdots.
\end{equation}
Therefore $H_{\fm}^{f}(Y)$ is a graded subspace of 
$H_{\fm}^{f}(Z)$, consequently, $\deg H_{\fm}^{f}(Y)\leq 
\deg H_{\fm}^{f}(Z)$ as desired. 

(4) Since $Y$ is a minimal free complex, the image of the map 
$Z\to X[1]$ is in $\fm X$. Define a complex $U$ where 
$U^{s}=X^{s}$ for all $1\leq s\leq w-2$ and 
$U^{w-1}=\coprod_{\sigma_{w-1,j}<\sigma_w} A(-\sigma_{w-1,j})$. 
Then $U$ is a subcomplex of $X$ and the image of the map 
$Z\to X[1]$ is in $U[1]$. Hence the composition 
$Z\to U[1]\to X[1]$ is the map $Z\to X[1]$. It suffices to show 
that the map $H^f_{\fm}(Z)_{\alpha}\to H_{\fm}^{f}(U[1])_{\alpha}$ 
is zero. By part (3), 
$$
\deg H_{\fm}^{f}(U[1]) \leq \deg (H_{\fm}^{f}(U^{w-1}[w]))
\leq \deg H_{\fm}^d(A) +(\sigma_{w}-1)=\alpha - 1.
$$
This means that $H_{\fm}^{f}(U[1])_{\alpha}=0$, and consequently, 
the map $H^f_{\fm}(Z)_{\alpha}\to H_{\fm}^{f}(U[1])_{\alpha}$ 
is zero.

(5) By definition, $\CMregxi(Z)=\CMregxi(F_w[w])$. When 
$\xi_1\ll 0$, 
$$\CMregxi(F_w[w])
=\xi_0\deg(H_{\fm}^d(A)+\sigma_w)+\xi_1 (-w+\depth(A))
=\xi_0 \alpha+\xi_1 f.$$
It is clear that, when $\xi_1\ll 0$, 
$$\max\limits_{0\leq s \leq w}
\{\CMregxi(F_s[s])\}=\CMregxi(F_w[w])=\CMregxi(Z).$$
By part (1), it remains to show that 
$\CMregxi(Y)\geq \CMregxi(Z)$. By part (4) and \eqref{E3.9.2},
$H^{f}_{\fm}(Y)_{\alpha}\cong H^{f}_{\fm}(Z)_{\alpha}$. 
By the definition of $\alpha$ and $f$, 
$H^{f}_{\fm}(Z)_{\alpha}\neq 0$. So $\deg H^{f}_{\fm}(Y)
\geq \alpha$ (in fact $=$). Therefore
$\CMregxi(Y)\geq \xi_0 \alpha+\xi_1 f=\CMregxi(Z)$
as desired.
\end{proof}

We are now ready to prove Theorem \ref{xxthm0.4}.

\begin{theorem}
\label{xxthm3.10}
Retain the above notation. Let $W$ be nonzero in 
$\D^{\bb}_{\fg}(A\Gr)$ with finite projective dimension. 
Suppose $\xi_0>0$.
\begin{enumerate}
\item[(1)]
If $0\leq \xi_1\leq \xi_0$, then 
$$\CMregxi(W)=\Torregxi(W)+\CMregxi(A).$$
\item[(2)]
If $\xi_1\ll 0$, then 
$$\CMregxi(W)=\Torregxi(W)+\CMregxi(A).$$
\end{enumerate}
\end{theorem}

Note that \cite[Theorem 2.8]{KWZ2} is a special 
case of Theorem \ref{xxthm3.10} by taking $\xi=(1,1)$.
Part of the proof of Theorem \ref{xxthm3.10} below 
is similar to the proof of \cite[Theorem 2.8]{KWZ2}.

\begin{proof}[Proof of Theorem \ref{xxthm3.10}] 
(1) By Lemma \ref{xxlem3.1}(1), we may assume that $\xi_0=1$.
So the condition $\xi_1\leq \xi_0$ becomes $\xi:=\xi_1\leq 1$.
By Lemma~\ref{xxlem3.1}(5, 6), we can assume that $\xi\neq 0,1$.

By Lemma \ref{xxlem3.1}(2), after a complex shift, we may 
assume that $W^n=0$ for all $n\geq 1$. Let $F$ be a minimal
free resolution of $W$, so it can be written as
$$F: \qquad \cdots \to 0\to F^{-s}\xrightarrow{d^{-s}} \cdots 
\to F^{-1} \xrightarrow{d^{-1}} F^0\to 0\to \cdots$$
for some $s\geq 0$. (Note that this $F$ is different from the 
complex introduced in \eqref{E3.7.1}.) We will prove the 
assertion by induction on $s$, which is 
the projective dimension of $W$. 

For the initial step, we assume that $s=0$, or 
$W= F^0= \bigoplus_{i} A(-a_i)$ for some integers $a_i$. 
In this case, it is clear that 
$$\Torregxi(W)=\Torregxi(F^0)=\max_i\{a_i\}=:a.$$
By Lemma \ref{xxlem3.1}(2),
$$\begin{aligned}
\CMregxi(W)& =\CMregxi\left(\bigoplus_{i} A(-a_i)\right)=
\max_i\left\{\CMregxi(A(-a_i))\right\}\\
&=\CMregxi(A)+\max_i\{a_i\}=\CMregxi(A)+\Torregxi(W).
\end{aligned}
$$ 
So the assertion holds for $W=F^0$ as required.

For the inductive step, assume that $s>0$. Let $F^{\leq -1}$ be 
the brutal truncation of the complex $F$
$$F^{\leq -1}:\qquad \cdots \to 0\to F^{-s}\to \cdots \to F^{-1} 
\to 0\to 0\to \cdots,$$
which is obtained by replacing $F^0$ with 0. We have a 
distinguished triangle in $\D^{\bb}_{\fg}(A\Gr)$
\begin{equation}
\label{E3.10.1}\tag{E3.10.1}
F^0\xrightarrow{\; f\; } F \to F^{\leq -1}\to F^0[1]
\end{equation}
where $F^0$ is viewed as a complex concentrated at position 0 
and $f$ is the inclusion. Let $G$ be the complex 
$F^{\leq -1}[-1]$, which is a minimal free complex concentrated 
in position $\{-(s-1), \cdots, 0\}$. Then we have a 
distinguished triangle in $\D^{\bb}_{\fg}(A\Gr)$
\begin{equation}
\label{E3.10.2}\tag{E3.10.2}
G\xrightarrow{\;\phi_2\;} 
F^0\xrightarrow{\; f\; }  F \to G[1]
\end{equation}
obtained by rotating \eqref{E3.10.1}. By the induction hypothesis,
the assertion holds for both $G$ and $F^0$. We need to show 
that the assertion holds for $W$, or equivalently, for $F$, as 
$F\cong W$ in $\D^{\bb}_{\fg}(A\Gr)$. By Theorem \ref{xxthm0.3}(2)
(=Theorem \ref{xxthm3.5}), it suffices to show  
\begin{equation}
\label{E3.10.3}\tag{E3.10.3}
\CMregxi(F)\geq \Torregxi(F)+\CMregxi(A).
\end{equation}

We fix the following temporary notation:
$$a=\Torregxi(F^0), \quad b=\Torregxi(G), \quad 
c=\Torregxi(F)=\Torregxi(W),$$
and 
$$\alpha=\CMregxi(F^0), \quad \beta=\CMregxi(G), \quad
\gamma=\CMregxi(F)=\CMregxi(W).$$
Note that $a$ is an integer. 
By definition, the minimality of $F$, 
and Lemma \ref{xxlem3.8}(1), we have
$$
c=\max\left\{ \Torregxi(F^0),\Torregxi(F^{\leq -1})\right\}
=\max\{a,b-\xi\}.
$$
Therefore, we have
\begin{equation}
\label{E3.10.4}\tag{E3.10.4}
a\leq c \quad
{\text{and}} \quad b - \xi \leq c.
\end{equation}
There are only three cases to consider:
\begin{enumerate}
\item[] \textbf{Case 1.} $c = a$ and $a \geq b$,
\item[] \textbf{Case 2.} $c = b - \xi$ (where  $0<\xi <1$ by 
                         the first paragraph), and  
\item[] \textbf{Case 3.} $c = a$ and $a < b$.
\end{enumerate}

\bigskip

\noindent\textbf{Case 1:} 
Suppose that $c = a$ and $a \geq b$. By the definition of $a$, 
we have $F^0=A(-a) \oplus C^0$ where $C^0$ is a graded left free 
$A$-module. Let $\phi_1: F^0\to A(-a)$ be the corresponding 
projection. By the definition of $\alpha:=\CMregxi(F^0)$, 
there is an integer  $j\in {\mathbb Z}$ such that 
$H^j_{\fm}(F^0)_{\alpha -\xi j}\neq 0$ and the induced
projection
$$\tau_1:=H^j_{\fm}(\phi_1)_{\alpha-\xi j}:
\quad H^j_{\fm}(F^0)_{\alpha -\xi j}\to 
H^j_{\fm}(A(-a))_{\alpha -\xi j}$$
is nonzero. The triangle \eqref{E3.10.2} gives rise to a long 
exact sequence 
\begin{equation}
\label{E3.10.5}\tag{E3.10.5}
\cdots 
\to  H^j_{\fm}(G)_{\alpha-\xi j}
\xrightarrow{\tau_2} H^j_{\fm}(F^0)_{\alpha -\xi j}
\to H^{j}_{\fm}(F)_{\alpha-\xi j}
\to H^{j+1}_{\fm}(G)_{\alpha-\xi j}
\to \cdots.
\end{equation}
Let $\phi_2: G\to F^0$ as in \eqref{E3.10.2}. If
$$\tau_2:=H^j_{\fm}(\phi_2)_{\alpha-\xi j}:
\quad H^j_{\fm}(G)_{\alpha-\xi j}
\to H^j_{\fm}(F^0)_{\alpha -\xi j}$$
is not surjective, then \eqref{E3.10.5} implies that 
$H^{j}_{\fm}(F)_{\alpha-\xi j}\neq 0$. By definition, the 
assumption that $a= c$, and the induction hypothesis, we have 
$$\begin{aligned}
\CMregxi(F) &\geq \alpha-\xi j +\xi j=\alpha\\
&=a+\CMregxi(A)= c+\CMregxi(A)\\
&=\Torregxi(F)+\CMregxi(A)
\end{aligned}$$
as desired. It remains to show that $\tau_2$ is not surjective. 
This follows from the following claim.

\bigskip

\noindent
{\bf Claim:} If $b<a+1$ (which covers Case 1), then $\tau_2:
=H^j_{\fm}(\phi_2)_{\alpha-\xi j}$ is not surjective.

\bigskip

\noindent
{\bf Proof of Claim:}
Assume to the contrary that $\tau_2$ is surjective. Then so is 
the composition
$$\tau_3:= \tau_1\circ \tau_2: \quad 
H^j_{\fm}(G)_{\alpha-\xi j}
\to H^j_{\fm}(A(-a))_{\alpha -\xi j}.$$
In particular, $\tau_3$ is not the zero map. Note that 
$$\tau_3= \tau_1\circ \tau_2=
H^j_{\fm}(\phi_1)_{\alpha-\xi j}
\circ H^j_{\fm}(\phi_2)_{\alpha-\xi j}
=H^j_{\fm}(\phi_1\circ \phi_2)_{\alpha-\xi j},$$
which implies that $\phi_3:=\phi_1\circ \phi_2$ is nonzero
in $\D^{\bb}_{\fg}(A\Gr)$. Consider $F$ as the cone of the map 
$\phi_2: G\to F^0$; it is clear that $\phi_2$ is the map from 
the top row $G$ to the middle row $F^0$ in the following 
diagram
$$\begin{CD}
F^{-s} @>>> \cdots @>>> F^{-2} @>>> F^{-1} @>>> 0\\
@V 0 VV @. @V 0VV @VV d^{-1}=\phi_2 V\\
0 @>>> \cdots @>>> 0 @>>> F^{0} @>>> 0\\
@V 0 VV @. @V 0VV @VV \phi_1 V\\
0 @>>> \cdots @>>> 0 @>>> A(-a) @>>> 0.
\end{CD}
$$
Note that $F^{-1}$ is the zeroth term in the minimal free 
resolution of $G$. Since $b< a+1$, $F^{-1}$ is generated in 
degree $< a+1$. Since $a$ is an integer, $F^{-1}$ is generated 
in degree $\leq a$. Since $F$ is a minimal free resolution 
of $W$, $\im \phi_2\subseteq \fm F^0$, and consequently, 
$\im \phi_3\subseteq \fm A(-a)$. For every generator $x$ in 
$F^{-1}$ which has degree $\leq a$, the image $\phi_3(x)$ lies 
in $\fm A(-a)$, which has degree at least $a+1$. Therefore 
$\phi_3(x)=0$. This implies that $\phi_3(F^{-1})=0$, yielding a 
contradiction. So we have proved the claim.

\bigskip

\noindent\textbf{Case 2:} Suppose $c = b - \xi$. Since $\xi>0$ 
by hypothesis, $c<b$. By the definition of $\beta:=\CMregxi(G)$, 
there is an integer $j\in {\mathbb Z}$ such that 
$H^j_{\fm}(G)_{\beta-\xi j}\neq 0$. The triangle \eqref{E3.10.2} 
gives rise to a long exact sequence
\begin{equation}
\label{E3.10.6}\tag{E3.10.6}
\cdots \to H^{j-1}_{\fm}(F)_{\beta-\xi j}\to
H^j_{\fm}(G)_{\beta-\xi j}
\to H^j_{\fm}(F^0)_{\beta -\xi j}\to \cdots.
\end{equation}
By the  induction hypothesis, the assumption that $c<b$, 
\eqref{E3.10.4}, and the definitions of $\alpha, \beta, a$ 
and $b$, we have
$$\begin{aligned}
\beta&=\CMregxi(G)=\Torregxi(G)+\CMregxi(A)
=b+\CMregxi(A)\\
&>c+\CMregxi(A)\geq a+\CMregxi(A)=\CMregxi(F^0)\\
&=\alpha,
\end{aligned}
$$
which implies that $H^j_{\fm}(F^0)_{\beta -\xi j}=0$. 
Since $H^j_{\fm}(G)_{\beta-\xi j}\neq 0$ by definition,
\eqref{E3.10.6} implies that $H^{j-1}_{\fm}(F)_{\beta-\xi j}\neq 0$.
By definition, $\CMregxi(F)\geq \beta-\xi j+\xi(j-1)=\beta-\xi$.
This inequality implies that
$$\begin{aligned}
\CMregxi(F)&\geq \beta-\xi=\CMregxi(G)-\xi\\
&=\Torregxi(G)+\CMregxi(A)-\xi=b+\CMregxi(A)-\xi\\
&=c+\CMregxi(A)=\Torregxi(F)+\CMregxi(A),
\end{aligned}
$$
as desired, see \eqref{E3.10.3}.

\bigskip

\noindent \textbf{Case 3:} Finally, suppose that $c = a$ and 
$a < b$. Since $c=\max\{a, b-\xi\}$, we must have $b-\xi
\leq c=a<b$. If $b-\xi=c$, the assertion follows 
from Case 2. It remains to show the assertion when 
$b-\xi<c=a<b$. Recall that $0<\xi<1$.

By the definition of $\beta:=\CMregxi(G)$, there is an integer 
$j\in {\mathbb Z}$ such that $H^j_{\fm}(G)_{\beta-\xi j}\neq 0$. 
Since $a<b$, by the  induction hypothesis, $\alpha<\beta$, 
consequently, $H^j_{\fm}(F^0)_{\beta -\xi j}=0$. By 
\eqref{E3.10.6}, we obtain that
$H^{j-1}_{\fm}(F)_{\beta-\xi j}\neq 0$. 
Therefore
$$\CMregxi(F)\geq \beta -\xi j +\xi(j-1)=\beta-\xi.$$

By the definition of $\alpha:=\CMregxi(F^0)$, there is another 
integer, still denoted by $j$, such that 
$H^j_{\fm}(F^0)_{\alpha-\xi j}\neq 0$. Since $b-\xi< a$ or 
$b<a+\xi$, by the induction hypothesis, $\beta <\alpha+\xi$, 
consequently, $H^{j+1}_{\fm}(G)_{\alpha -\xi j}=0$. 
Since $b<a+\xi<a+1$, by Claim, $\tau_2$ is not surjective.
By \eqref{E3.10.5}, we obtain that 
$H^{j}_{\fm}(F)_{\alpha-\xi j}\neq 0$. Therefore
$$\CMregxi(F)\geq \alpha -\xi j +\xi j=\alpha.$$

Combining the two previous inequalities, using the induction 
hypothesis and the fact that $c=\max\{a, b-\xi\}$, we obtain that
$$\begin{aligned}
\CMregxi(F)&\geq \max\{\alpha, \beta-\xi\}\\
&=\max\{a +\CMregxi(A), b-\xi +\CMregxi(A)\}\\
&=\max\{a,b-\xi\}+\CMregxi(A)\\
&=c+\CMregxi(A)=\Torregxi(F)+\CMregxi(A)
\end{aligned}
$$
as desired.

Combining these three cases completes the proof.

\bigskip

(2) We will use the notation introduced before Lemma 
\ref{xxlem3.8}. 

First of all, when $W=A$, we have
$$\begin{aligned}
\CMregxi(W)&=\CMregxi(A)=0+\CMregxi(A)
=\Torregxi(A)+\CMregxi(A)\\
&=\Torregxi(W)+\CMregxi(A),
\end{aligned}
$$
so the assertion holds for $W=A$. By Lemma \ref{xxlem3.1}(1),
the assertion holds for 
$W=Z:=[\coprod_{j=1}^{n_w} A(-\sigma_{w,j})][w]$. When 
$\xi_1\ll 0$, we have, by setting $W=Y$, 
$$\begin{aligned}
\CMregxi(W)&=\CMregxi(Y)
=\CMregxi(Z) \quad {\text{by Lemma \ref{xxlem3.9}(5)}}\\
&=\Torregxi(Z)+\CMregxi(A)\\
&=\Torregxi(Y)+\CMregxi(A) 
\; \quad {\text{by Lemma \ref{xxlem3.8}(2)}}\\
&=\Torregxi(W)+\CMregxi(A)
\end{aligned}
$$
as desired.
\end{proof}

\begin{remark}
\label{xxrem3.11} 
Retain the hypotheses of Theorem~\ref{xxthm3.10}. 
\begin{enumerate}
\item[(1)]
If $\xi_1>\xi_0>0$, the conclusion of Theorem~\ref{xxthm3.10}(1),
\begin{equation}
\label{E3.11.1}\tag{E3.11.1}
\CMregxi(W)=\Torregxi(W)+\CMregxi(A)
\end{equation}
may fail to hold for some $W$. For example, let $A$ be a 
noetherian Koszul AS regular algebra of type $(d,\bfl)$ with 
$d=\bfl\geq 1$. By Example \ref{xxex2.2}(2), 
$\CMregxi(A)=\xi_1 d-\xi_0\bfl=(\xi_1-\xi_0)d$. It is easy to 
check that $\Torregxi(\Bbbk)=0$ when $\xi_1>\xi_0>0$. By 
Example \ref{xxex2.2}(1), $\CMregxi(\Bbbk)=0$. Then 
$$\CMregxi(\Bbbk)=0< (\xi_1-\xi_0)d
=\Torregxi(\Bbbk)+\CMregxi(A).$$
\item[(2)]
By Theorem \ref{xxthm3.10}(2) \eqref{E3.11.1} holds for all 
$\xi_1 \ll 0$. It is unknown if equation \eqref{E3.11.1}
holds for all $\xi_1<0$. 
\end{enumerate}
\end{remark}

Next we recover the  theorem of Auslander and Buchsbaum 
\cite[Theorem 3.2]{Jo2} as a special case of Theorem 
\ref{xxthm3.10}.

\begin{corollary}
\label{xxcor3.12} 
Let $A$ be a noetherian connected graded algebra with balanced
dualizing complex. Let $X$ be a nonzero object in 
$\D^{\bb}_{\fg}(A\Gr)$ with finite projective dimension. 
\begin{enumerate}
\item[(1)]{\rm [}The Auslander--Buchsbaum Formula{\rm{]}}
$$\pdim (X)+\depth(X)=\depth(A).$$
\item[(2)]
Let $p(X) = \pdim (X)$ and $d(X) = \depth(X)$. Then
\begin{equation}
\notag
\deg H^{d(X)}_{\fm}(X)
=\deg \Tor^A_{p(X)}(\Bbbk, X)+\deg H^{d(A)}_{\fm}(A).
\end{equation}
\end{enumerate}
\end{corollary}

\begin{proof} 
Let $\xi_0>0$ be fixed. When $\xi_1\ll 0$, by definition,
$$\begin{aligned}
\CMregxi(X)
&=\sup_{i\in {\mathbb Z}} \{\xi_0 \deg(H^i_{\fm}(X))+\xi_1 i\}\\
&=\xi_0 \deg(H^{d(X)}_{\fm}(X))+\xi_1 d(X),\\
\Torregxi(X)
&=\sup_{i\in {\mathbb Z}}
   \{ \xi_0 \deg(\Tor^A_{i}(\Bbbk, X))-\xi_1 i\}\\
&=\xi_0 \deg(\Tor^A_{p(X)}(\Bbbk, X))-\xi_1 p(X).
\end{aligned}
$$
As a special case of the first equation, when $\xi_1\ll 0$,
$$\CMregxi(A)=\xi_0 \deg(H^{d(A)}_{\fm}(A))+\xi_1 d(A).$$

By Theorem \ref{xxthm3.10}, we have, for all $\xi_1\ll 0$,
$$\CMregxi(X)=\Torregxi(X)+\CMregxi(A),$$
or equivalently,
{\small $$\xi_0 \deg(H^{d(X)}_{\fm}(X))+\xi_1 d(X)
=\xi_0 \deg(\Tor^A_{p(X)}(\Bbbk, X))-\xi_1 p(X)
+\xi_0 \deg(H^{d(A)}_{\fm}(X))+\xi_1 d(A).$$}
Therefore both parts (1) and (2) follow 
from the above equation by comparing the coefficients on
$\xi_0$ and $\xi_1$ \textcolor{blue}{as $\xi_1$ varies.}.
\end{proof}

To conclude this section, we generalize \cite[Theorem 3.1]{Jo4} 
by removing the Koszul assumption.  As noted after 
Definition~\ref{xxdef2.1}, if $M$ is a finitely generated left 
$A$-module then $\CMregxi(M)$ is finite.

\begin{theorem}
\label{xxthm3.13}
Let $\xi:=(1, \xi_1)$, $\epsilon:=\max\{0, \xi_1-1\}$, and 
$c:=\Torregxi(\Bbbk)$.
Let $M$ be a nonzero finitely generated graded left $A$-module.
Then for every integer $s\geq \CMregxi(M)$, 
we have
$$i \leq \ged \Tor^A_i(\Bbbk, M_{\geq s}(s)) 
\leq \deg \Tor^A_i(\Bbbk, M_{\geq s}(s)) 
     \leq c+\epsilon+ i \xi_1$$
whenever $\Tor^A_i(\Bbbk, M_{\geq s}(s))\neq 0$.
\end{theorem}

\begin{proof} Since $M_{\geq s}(s)$ is a module generated by 
elements of non-negative degrees, it is clear that
$i \leq \ged \Tor^A_i(\Bbbk, M_{\geq s}(s))$ when 
$\Tor^A_i(\Bbbk, M_{\geq s}(s))\neq 0$. It remains 
to show that if $s \geq \CMregxi(M)$, then
$\deg \Tor^A_i(\Bbbk, M_{\geq s}(s)) \leq c+ \epsilon+i \xi_1$
for all $i$, or equivalently that $\Torregxi(M_{\geq s}(s))
\leq c+\epsilon$.

Let $s\geq \CMregxi(M)$. It is clear that
$\deg M/M_{\geq s}\leq s-1$. By Example~\ref{xxex2.2}(1),
$$\CMregxi(M/M_{\geq s})=\deg M/M_{\geq s} \leq s-1.$$
Applying Lemma \ref{xxlem3.7}(2) to the short 
exact sequence
$$0\to M_{\geq s}\to M\to M/M_{\geq s}\to 0,$$
we obtain that $\CMregxi(M_{\geq s})\leq \max\{s, s-1+\xi_1\}
=s+\epsilon$. By Theorem \ref{xxthm3.3}, 
$$\Torregxi(M_{\geq s})
\leq \CMregxi(M_{\geq s})+\Torregxi(\Bbbk)\leq s+\epsilon+c.$$
By Lemma \ref{xxlem3.1}(2) we
obtain that $\Torregxi(M_{\geq s}(s))\leq c+\epsilon$ 
as desired.
\end{proof}

If $A$ is Koszul then $c = \Torregxi(\Bbbk) =0$, and if, 
in addition, $\xi_1 =1$ then Theorem \ref{xxthm3.13} says
that $M_{\geq s}(s)$ has a linear resolution,
which recovers \cite[Theorem 3.1]{Jo4}.

\section{Weighted Artin--Schelter Regularity}
\label{xxsec4}

Extending Definition \ref{xxdef0.2}(3) to any 
$\xi= (\xi_0, \xi_1)$ the weighted AS-regularity 
(or $\xi$-AS regularity) of $A$ is defined to be 
$$\ASregxi(A)=\Torregxi(\Bbbk)+\CMregxi(A).$$

In this section we prove results that are related to AS 
Gorenstein and AS regular algebras. We begin with a 
generalization of a nice result of Dong and Wu 
\cite[Theorem 4.10]{DW} that provides a proof of part of 
Theorem \ref{xxthm0.6}, showing that (1) and (3) are 
equivalent; parts (1) and (2) are equivalent by 
\cite[Theorem 0.8]{KWZ2}. Note that by Remark \ref{xxrem2.6}(1,3), 
the existence of $\xi$ such that $\ASregxi(A)=0$ in part (ii) 
of the theorem below does not imply that $\ASreg_{\zeta}(A)=0$ 
for all $\zeta$.

\begin{theorem}
\label{xxthm4.1}
Let $A$ be a noetherian connected graded algebra with balanced
dualizing complex. Then the following are equivalent:
\begin{enumerate}
\item[(i)]
$A$ is AS regular.
\item[(ii)]
$A$ is Cohen--Macaulay and 
there exists a $\xi=(\xi_0,\xi_1)$ with $\xi_0>0$ 
such that $\ASregxi(A)=0$. 
\end{enumerate}
\end{theorem}

An AS Gorenstein algebra is called \emph{standard} if $\bfl = d$.
When $A$ is Koszul, then \cite[Theorem 4.10]{DW} can be 
recovered from Theorem \ref{xxthm4.1} since, by definition, 
standard AS Gorenstein algebras satisfy (ii) in the above 
theorem. Note that \cite[Theorem 3.2]{KWZ2} is a 
special case of Theorem \ref{xxthm4.1} by taking $\xi=
(1,1)$. Our proof of Theorem \ref{xxthm4.1} below 
is very close to the proof of \cite[Theorem 3.2]{KWZ2}. 

\begin{proof}[Proof of Theorem \ref{xxthm4.1}]
We first prove that (i) implies (ii). Suppose that $A$ is AS 
regular of type $(d,\bfl)$. It is well-known that $A$ is 
Cohen--Macaulay. If we let $\xi=(1,1)$, then 
$$\CMregxi(A)=d-\bfl=-\Torregxi(\Bbbk)$$
by Remark \ref{xxrem2.6}(3), and so part (ii) holds.

We now show that (ii) implies (i). Let $A$ be noetherian 
connected graded with balanced dualizing complex. If 
$\pdim \Bbbk<\infty$, then $A$ has finite global dimension. 
Since $A$ is noetherian, if it has finite global dimension, 
then it has finite GK dimension. By \cite[Theorem 0.3]{Zh},
$A$ is AS Gorenstein and so $A$ is AS regular by definition. 
Hence, it suffices to show that $\pdim \Bbbk<\infty$. 

Let
\begin{equation}
\label{E4.1.1}\tag{E4.1.1}
F: \quad \cdots \to F_{i}\to \cdots \to F_0\to 
{_A\Bbbk}\to 0
\end{equation}
be a minimal free resolution of the trivial left $A$-module
$_A\Bbbk$. Since $A$ is Cohen--Macaulay, by 
\cite[Theorem 6.3]{VdB}, the balanced dualizing complex over 
$A$ is 
$$R\cong \R\Gamma_{\fm}(A)'\cong \omega[d]$$
where $\omega$ is a dualizing $A$-bimodule and 
$d={\rm{lcd}}(A)$. By Theorem \ref{xxthm1.4}(1), for every 
complex $X$ of left graded $A$-modules,
\begin{equation}
\label{E4.1.2}\tag{E4.1.2}
\R\Gamma_{\fm}(X)'\cong \R\Hom_A(X, R)
\cong \R\Hom_A(X,\omega[d]).
\end{equation}
Since the dualizing complex has finite injective dimension, 
\eqref{E4.1.2}, taking $X$ to be an $A$-module $M$, implies that
$$d=\injdim (\omega)<\infty.$$
As a consequence of \eqref{E4.1.2}, $\Gamma_{\fm}$ has 
cohomological dimension $d$.

For each $j\geq 0$, let $Z_j(F)$ denote the $j$th syzygy of the 
complex $F$ \eqref{E4.1.1}. We will show that $Z_j(F)=0$ for 
$j\gg 0$, which implies that $\pdim \Bbbk<\infty$ as desired. 
Assume to the contrary that there is an increasing sequence
$j_1<j_2<\cdots$ such that $Z_{j_s}(F)\neq 0$ for all $s\geq 1$. 
Then $Z_j(F)\neq 0$ for all
$j\geq 0$. Note that
$$\cdots \to F_{j+2}\to F_{j+1}\to Z_j(F)\to 0$$
is a minimal free resolution of $Z_j(F)$.  

\bigskip

\noindent 
\textbf{Claim.} For all $j \geq d$, $t^A_{j+1}(\Bbbk) \leq 
t^A_{j}(\Bbbk)$.

\bigskip

\noindent
{\it Proof of the claim.}
By the balanced dualizing complex condition, 
$\Ext^i_A(\Bbbk, \omega)=0$ for all $i\neq d=\injdim \omega$. By 
induction on syzygies, we have $\Ext^i_A(Z_{d-1}(F),\omega)=0$ 
for all $i\neq 0$. Further, by induction, one sees that 
$\Ext^i_A(Z_{j-1}(F), \omega)=0$ for all $i\neq 0$ and all 
$j\geq d$. From now on, we fix $j\geq d$. By \eqref{E4.1.2}, we 
obtain that $H^i_{\fm}(Z_{j-1}(F))=0$ for all $i\neq d$. Since $A$
is Cohen--Macaulay, $H^i_{\fm}(F_j)=0$ for all $i\neq d$.
Applying $\R\Gamma_{\fm}(-)$ to the short exact sequence
$$0\to Z_{j}(F)\to F_j\to Z_{j-1}(F)\to 0,$$
we obtain a long exact sequence, which has only three 
nonzero terms yielding a short exact sequence
$$0\to H^d_{\fm}(Z_j(F))\to H^d_{\fm}(F_j)\to H^d_{\fm}(Z_{j-1}(F))
\to 0.$$
The above short exact sequence implies that 
$\deg H^d_{\fm}(Z_j(F))\leq \deg H^d_{\fm}(F_j)$.
By definition, 
\begin{equation}
\label{E4.1.3}\tag{E4.1.3}
\CMregxi(Z_j(F))\leq \CMregxi(F_j).
\end{equation}
By Corollary \ref{xxcor3.6}, for any $X\in \D^{\bb}_{\fg}(A\Gr)$
\begin{equation}
\label{E4.1.4}\tag{E4.1.4}
\Torregxi(X)=\CMregxi(X)+c,
\end{equation}
where $c = -\CMregxi(A)$. Then 
$$\begin{aligned}
t^A_{j+1}(\Bbbk)&= t^A_{0}(F_{j+1}) = t^A_{0}(Z_j(F))\\
&= \frac{1}{\xi_0} (\xi_0 t^A_{0}(Z_j(F))-\xi_1 0)\\
& \leq \frac{1}{\xi_0} \sup \{\xi_0 t^A_{i}(Z_j(F))-
\xi_1 i\mid i\in {\mathbb Z}\}\\
&=\frac{1}{\xi_0} \Torregxi(Z_j(F))\\
&=\frac{1}{\xi_0} (\CMregxi(Z_j(F))+c) \quad\;\quad 
{\text{by \eqref{E4.1.4}}}\\
&\leq \frac{1}{\xi_0} (\CMregxi(F_j)+c)\qquad \qquad
{\text{by \eqref{E4.1.3}}}\\
&=\frac{1}{\xi_0} \Torregxi(F_j)\\
&=\frac{1}{\xi_0} \sup \{\xi_0 t^A_{i}(F_j)-
\xi_1 i\mid i\in {\mathbb Z}\}\\
&=\frac{1}{\xi_0}(\xi_0 t^A_{0}(F_j)-
\xi_1 0)=t^A_{0}(F_j)\\
&=t^A_{j}(\Bbbk)
\end{aligned}
$$
as desired. This finishes the proof of the claim.

Since $A$ is connected graded and since $F$ is the minimal
free resolution of $_A \Bbbk$, we have
\begin{equation}
\label{E4.1.5}\tag{E4.1.5}
t^A_{j}(\Bbbk)\geq j
\end{equation} 
whenever $F_j\neq 0$. Then for 
$j\gg 0$, the claim contradicts \eqref{E4.1.5}.
Therefore we obtain a contradiction, 
and hence $\pdim \Bbbk<\infty$ as required. 
\end{proof}

\begin{remark}
\label{xxrem4.2}
Suppose \cite[Hypothesis 2.7]{KWZ1} holds. Then, by Example 
\ref{xxex2.2}(3), $\CMregxi(A)=\xi_0 \deg_t h_A(t) +\xi_1 d$. 
Then (ii) is equivalent to 
$$\Torregxi(\Bbbk)=-(\xi_0 \deg_t h_A(t) +\xi_1 d)$$ 
which could be easier to compute than computing these 
regularities from their definitions.
\end{remark}

\begin{proof}[Proof of Theorem \ref{xxthm0.6}]
(1) $\Leftrightarrow$ (3) This is Theorem \ref{xxthm4.1}.

(1) $\Rightarrow$ (2) See the first paragraph of the proof 
of Theorem \ref{xxthm4.1}.

(2) $\Rightarrow$ (1) By the second paragraph of the proof 
of Theorem \ref{xxthm4.1}, it suffices to show that
$\pdim \Bbbk<\infty$. If $\xi=1$, the assertion follows from 
\cite[Theorem 0.8]{KWZ2}. Note that the hypothesis 
$\ASregxi(A)=0$ implies that $\Torregxi(\Bbbk)<\infty$.
If $\xi<1$, it follows from \eqref{E2.4.1} that 
$\pdim \Bbbk<\infty$ as desired.
\end{proof}

\begin{proof}[Proof of Corollary \ref{xxcor0.7}]
By definition, the hypothesis implies that 
$\Torregxi(\Bbbk)\leq \bfl- \xi d=-\CMregxi(A)$.
Thus $\ASregxi(A)\leq 0$. By Theorem \ref{xxthm0.3}(3),
$\ASregxi(A)=0$. Now the assertion follows from 
Theorem \ref{xxthm0.6}(3$\Rightarrow$1).
\end{proof}

The following result is a weighted and noncommutative (and 
non-generated in degree 1) version of 
\cite[Theorem 1.3(i$\Leftrightarrow$ii$\Leftrightarrow$v)]{Rom}.

\begin{theorem}
\label{xxthm4.3}
Let $A$ be a noetherian connected graded algebra with balanced
dualizing complex. Then the following are equivalent:
\begin{enumerate}
\item[(1)]
$A$ is AS regular.
\item[(2)]
There exists a pair $\xi=(\xi_0,\xi_1)$ with $\xi_0>0$ 
and $\xi_1\leq \xi_0$ such that
$$\CMregxi(X)=\Torregxi(X)+\CMregxi(A)$$
for all $0\neq X\in \D^{\bb}_{\fg}(A\Gr)$.
\item[(3)]
There exists a pair $\xi=(\xi_0,\xi_1)$ with $\xi_0>0$ 
and $\xi_1\leq \xi_0$ such that
$$\CMregxi(M)=\Torregxi(M)+\CMregxi(A)$$
for all noetherian modules $0\neq M\in A\Gr$.
\item[(4)]
There exist $c$ and a pair $\xi=(\xi_0,\xi_1)$ with $\xi_0>0$ 
and $\xi_1\leq \xi_0$ such that
$$\CMregxi(X)=\Torregxi(X)-c$$
for all $0\neq X\in \D^{\bb}_{\fg}(A\Gr)$.
\item[(5)]
There exist $c$ and a pair $\xi=(\xi_0,\xi_1)$ with $\xi_0>0$ 
and $\xi_1\leq \xi_0$ such that
$$\CMregxi(M)=\Torregxi(M)-c$$
for all noetherian modules $0\neq M\in A\Gr$.
\end{enumerate}
\end{theorem}

\begin{proof}
(1) $\Rightarrow$ (2)
This follows from Theorem \ref{xxthm4.1} and 
Corollary \ref{xxcor3.6}.

(2) $\Rightarrow$ (3), (2) $\Rightarrow$ (4),
(2) $\Rightarrow$ (5), (3) $\Rightarrow$ (5),
(4) $\Rightarrow$ (5) Trivial.

(5) $\Rightarrow$ (1) 
Let $M=A$, then $\CMregxi(A)=\Torregxi(A)-c=-c$.
Therefore $c=-\CMregxi(A)$. Let $M=\Bbbk$. Then 
$0=\CMregxi(\Bbbk)=\Torregxi(\Bbbk)-c$ 
which implies that $c=\Torregxi(\Bbbk)$. Therefore
$\ASregxi(A)=\CMregxi(A)+\Torregxi(\Bbbk)=-c+c=0$.
Now the assertion follows from Theorem \ref{xxthm0.6}.
\end{proof}

\begin{remark}
\label{xxrem4.4}
By Remark \ref{xxrem2.6}(1), if $A$ is AS regular, 
$\Torregxi(\Bbbk)$ may not be zero. So 
\cite[Theorem 1.3(i$\Leftrightarrow$iv)]{Rom}
cannot be generalized to the noncommutative case.
\end{remark}

The next two theorems are weighted versions of 
\cite[Theorem 5.1]{Jo3} and \cite[Corollary 5.2]{Jo3}. In 
the following theorems and their proofs, we use both
$$0\to F^{-s}\to \cdots \to F^0\to 0$$ 
and 
$$0\to F_{s}\to \cdots \to F_0\to 0$$ 
to denote the same complex $F$ after identifying $F^{-i}$ 
with $F_{i}$. Similarly, we use both $H^{-i}(F)$ and 
$H_{i}(F)$ for the same (co)homology. 

Suppose that $F:=0\to F^{-s}\to \cdots \to F^0\to 0$ is a 
minimal complex of finitely generated free $A$-modules bounded 
left at position zero. Application of the functor 
$(-)^\vee:=\Hom_A(-,A)$ on $F$ yields the minimal free complex 
$F^\vee$ that is bounded above at position zero. 

\begin{theorem}
\label{xxthm4.5}
Retain the above notation. Let $\xi \leq 1 $ be a real number. 
Then, for all integers $0\leq c\leq s$, we have
$$\max_{0 \leq j \leq c}
\left\{ -\ged (H^{s-j}(F^\vee))+ \xi j\right\}
=\max_{0 \leq j \leq c}
\left\{\deg (H_{s-j}(\Bbbk\otimes_A F))+\xi j \right\}.$$
\end{theorem}

\begin{proof}
We proceed by induction on $c\geq 0$. 

For the initial step, let $c=0$. If $F^{-s}=0$, then both sides 
of the equation are $-\infty$, so we may assume that 
$F^{-s}\neq 0$. The complex $F^\vee$ is minimal, and the top 
term in this complex is $(F^\vee)^s=(F^{-s})^\vee=(F_s)^\vee$. 
Therefore
$$-\ged H^s(F^\vee)=-\ged (F^\vee)^s=-\ged (F_s^\vee),$$
which is equal to the maximal degree of the generators of $F_s$. 
But by the minimality of the complex this is again equal to 
$\deg H_s(\Bbbk\otimes_A F)$, so
$$-\ged H^s(F^\vee) +\xi 0= \deg H_s(\Bbbk\otimes_A F) +\xi 0,$$
and the equation holds for $c=0$. 

For the inductive step, we employ the following notation: for 
an integer $e \geq 0$, let
$$\begin{aligned}
x_{e}&=\max_{0 \leq j \leq e}
\left\{ -\ged (H^{s-j}(F^\vee))+ \xi j \right\}\\
y_{e}&=\max_{0 \leq j \leq e}\left\{\deg (H_{s-j}
(\Bbbk\otimes_A F))+\xi j \right\},
\end{aligned}
$$
and suppose that $x_e=y_e$ for every $e\leq c$. Set
$$\begin{aligned}
\alpha&=
{\text{maximal degree of a generator in $F_{s-c}=F^{-s+c}$}},\\
\beta&=
{\text{maximal degree of a generator in $F_{s-c-1}=F^{-s+c+1}$}}.
\end{aligned}
$$
There are two cases to consider. 

\noindent\textbf{Case 1.} 
Suppose $\beta\leq \alpha-1$. In this case, using the fact that 
$\xi\leq 1$, we obtain that 
$$y_{c+1}=\max\{y_{c}, \beta+ \xi(c+1)\}
\leq \max\{y_{c}, \beta+ 1+\xi c\}
\leq \max\{y_{c}, \alpha+\xi c\}=y_{c}.$$
By the induction hypothesis and the definition of $y_c$, we have 
$x_{c}=y_{c}\geq \alpha +\xi c$. Since 
$\ged H^{s-c-1}(F^\vee)\geq \ged (F^\vee)^{s-c-1}$, we obtain
that
$$\begin{aligned}
- \ged H^{s-c-1}(F^\vee)+\xi (c+1)
&\leq -\ged (F^\vee)^{s-c-1} + \xi(c+1)\\
&= \beta +\xi(c+1) \leq \alpha+ \xi c\leq x_{c},
\end{aligned}
$$
whence $x_{c+1}=x_c$. Therefore, $x_{c+1} = x_c = y_c = y_{c+1}$. 

\noindent\textbf{Case 2:} 
Suppose $\beta\geq \alpha$. In this case we have 
$y_{c+1}=\max\{ y_{c}, \beta+\xi(c+1)\}$ as before. On the 
other hand, $F^\vee$ is minimal, so since $-\beta$ is the 
minimal degree of the generators of $(F^\vee)^{s-c-1}$ and 
$-\alpha$ is the minimal degree of the generators of 
$(F^\vee)^{s-c}$, the inequality $-\beta\leq -\alpha$ implies 
that any element in the minimal degree of $(F^\vee)^{s-c-1}$ 
is mapped to zero in $(F^\vee)^{s-c}$ (since it has to have 
image inside $\fm (F^\vee)^{s-c}$, and this module begins in 
degree $-\alpha+1$). By minimality of the complex, elements 
in the minimal degree of $(F^\vee)^{s-c-1}$ are not in the 
image of the differential, hence $\ged H^{s-c-1}(F^\vee)=-\beta$. 
Now we obtain that
$$x_{c+1}=\max\{x_{c}, -\ged H^{s-c-1}(F^\vee)+
\xi(c+1)\}=\max\{x_{c},\beta+\xi(c+1)\},$$
which implies that $x_{c+1}=y_{c+1}$. By induction, we have 
proved the claim.
\end{proof}

The following theorem can be used to compute the weighted 
CM regularity of a finitely generated nonzero graded module 
over an AS Gorenstein algebra.

\begin{theorem}[Theorem  \ref{xxthm0.8}]
\label{xxthm4.6} 
Suppose $A$ is a noetherian, connected graded AS Gorenstein
algebra of type $(d,\bfl)$. Let $\xi \leq 1$ be a real number 
{\rm{(}}and also by abuse the notation let $\xi=(1,\xi)${\rm{)}}. 
Let $M\neq 0$ be a finitely generated left graded $A$-module 
with finite projective dimension.
\begin{enumerate}
\item[(1)]
Let $w$ be an integer with $0 \leq w \leq d$. Then
\[
\max_{0 \leq j \leq w}\left\{\deg H^j_{\fm}(M)  +\xi j \right\} 
=-\bfl+\xi d+\max_{d-w \leq j \leq d}
\left\{ \deg \Tor^A_{j} (\Bbbk, M) -\xi j \right\}.
\]
\item[(2)]
In particular, if $0 \leq w \leq d$ is the maximum integer such 
that $H^w_{\fm}(M)\neq 0$, we have
$$\CMregxi(M)=-\bfl+\xi d+\max_{d-w \leq j \leq d}
\left\{\deg \Tor^A_j(\Bbbk,M)-\xi j \right\}.$$
\item[(3)]
If, further, $M$ is $s$-Cohen--Macaulay, then 
$$\CMregxi(M)=-\bfl+\xi s+\deg(\Tor^A_{d-s}(\Bbbk, M)).$$
\end{enumerate}
\end{theorem}

\begin{proof}
(1) Let $p$ be the projective dimension of $M$.
By the Auslander--Buchsbaum Formula, $p \leq \depth(A) = d$.
Suppose $0\to F_{p}\to \cdots \to F_0\to M\to 0$ is a minimal 
free resolution of $M$ and write
$$ F:= \qquad \quad \cdots \to 0\to F_{p}\to \cdots 
\to F_0\to 0 \cdot $$
(where we have removed the $M$ term). By Theorem \ref{xxthm1.4}(1) 
and the hypothesis that $A$ is AS Gorenstein of type $(d, \bfl)$, 
for all integers $i$, we have 
$$H^i_{\fm}(M)\cong \Ext^{d-i}_A(M,A)'(\bfl).$$
Note that $\Ext^j_A(M,A)$ can be computed by using the 
complex $F^{\vee}:=\Hom_A(F,A)$. Then we have
$$\begin{aligned}
\max_{0 \leq j \leq w}&\{\deg(H^j_{\fm}(M))+\xi j\} 
= \max_{0 \leq j \leq w} \left\{ \deg \left(
\Ext_A^{d-j}(M,A)'\right) - \bfl + \xi j \right\}\\
&=\max_{0 \leq j \leq w}\left\{-\ged(H^{d-j}(F^{\vee}))
 -\bfl+\xi j \right\}\\
&=-\bfl+\xi d +\max_{0 \leq j \leq w}
 \left\{-\ged(H^{d-j}(F^{\vee}))-\xi d+\xi j \right\}\\
&= -\bfl+\xi d +\max_{d-w \leq k \leq d}
 \left\{-\ged(H^{k}(F^{\vee}))-\xi k  \right\} \\
&= -\bfl+\xi d +\max_{d-w \leq k \leq p}
 \left\{-\ged(H^{k}(F^{\vee}))-\xi k  \right\}  
 \quad \text{since $\pdim M = p$}\\
&=-\bfl+\xi d +\max_{0 \leq i \leq p - d + w}
\left\{-\ged(H^{p-i}(F^{\vee}))-\xi p+\xi i  \right\}\\
&=-\bfl+\xi d -\xi p+\max_{0 \leq i \leq p - d + w}
\left\{-\ged(H^{p-i}(F^{\vee})) +\xi i  \right\}.
\end{aligned}
$$
It is clear that
$$\begin{aligned}
\max_{d-w \leq j \leq d}&\left\{\deg 
(\Tor^A_j(\Bbbk,M)) -\xi j \right\}
=\max_{p-d \leq i \leq p - d + w}
\left\{\deg (\Tor^A_{p-i}(\Bbbk,M))  -\xi p+\xi i\right \} \\
&=\max_{ 0 \leq i \leq p - d + w}\{\deg (\Tor^A_{p-i}(\Bbbk,M)) 
-\xi p+\xi i\}\\
&=-\xi p +\max_{ 0 \leq i \leq p - d + w}
\{\deg (\Tor^A_{p-i}(\Bbbk,M))  +\xi i\}.
\end{aligned}
$$
Now by Theorem \ref{xxthm4.5}, we have
\[
\max_{0\leq i\leq p-d+w}
  \left\{-\ged(H^{p-i}(F^{\vee}))+\xi i \right\}
=\max_{0 \leq i\leq p-d+w}\left\{\deg (\Tor^A_{p-i}(\Bbbk,M)) 
+\xi i\right\}.
\]
Therefore the assertion follows.

Part (2) is a special case of part (1).

For part (3), if $M$ is $s$-Cohen--Macaulay, then $s=\depth(M)$, 
and by the Auslander--Buchsbaum Formula, 
$\pdim(M)=\depth(A)-\depth(M)=d-s$ is the projective dimension 
of $M$. Taking $w=s=d-p$, it follows from part (2) that
$$\begin{aligned}
\CMregxi(M)&= -\bfl+\xi d +
\max_{d-s\leq j\leq d}\{\deg \Tor^A_{j}(\Bbbk,M) -\xi j\}\\
&= -\bfl+\xi d +\max_{d-s\leq j\leq d-s}
\{\deg \Tor^A_{j}(\Bbbk,M) -\xi j\} \\
&\qquad\qquad\qquad \qquad\qquad \quad 
\text{ since $\pdim(M) = d -s$}\\
&=-\bfl+\xi d +\deg \Tor^A_{d-s}(\Bbbk,M) -\xi (d-s)\\
&=-\bfl+\xi s +\deg \Tor^A_{d-s}(\Bbbk,M)
\end{aligned}
$$
as desired.
\end{proof}

\begin{remark}
\label{xxrem4.7}
Retain the hypothesis as Theorem \ref{xxthm4.6}(3). 
Let $p$ be the projective dimension of $M$. 
\begin{enumerate}
\item[(1)]
Recall that $t^j(M)=\deg \Tor^A_j(\Bbbk,M)$ for all $i$.
Note that $\CMregxi(M) = \deg H^s_{\fm}(M) + \xi s$. 
It follows from Theorem 4.6(3) that 
$$t^p(M)=\deg H_{\fm}^s(M)+\bfl$$
which is a special case of Corollary \ref{xxcor3.12}(2).
\item[(2)]
By Theorem \ref{xxthm4.6}(3, 2) (taking $w=d$), 
we obtain that 
$$\begin{aligned}
t^p(M)-\xi p&=\CMregxi(M)+\bfl -\xi s-\xi p\\
&=\max_{0 \leq j \leq d}\left\{ t^j(M) -\xi j \right\}
\end{aligned}
$$
for all $\xi\leq 1$. Therefore, for each $j$,
$t^j(M) -\xi j\leq t^p(M)-\xi p$ for all $\xi\leq 1$. 
By taking $\xi=1$, we have $t^p(M)-t^j(M)\geq p-j$ 
for all $0\leq j\leq p$.
\end{enumerate}
\end{remark}

\section{Related invariants}
\label{xxsec5}
In this section we consider concavity, rate, and slope, 
homological invariants that are related to our weighted 
regularities and to homological invariants that have been 
studied in the literature.

\subsection{Concavity}
\label{xxsec5.1}
In this subsection we use the letters $A$ or $B$ for connected 
graded noetherian algebras, $S$ or $T$ for connected graded 
noetherian AS regular algebras, and $G$ for a general locally 
finite graded noetherian algebra.

In \cite[Definition 0.9]{KWZ2}, we introduced the notion of the
\emph{concavity} of a numerical invariant $\mathcal{P}$. We 
recall the definition here. A graded algebra map $\phi: A \to G$ 
is called \emph{finite} if $_AG$ and $G_A$ are finitely generated.
For a locally finite graded noetherian algebra $G$, let 
$$\Phi(G)=\{T \mid {\text{$T$ is AS regular and there is a finite 
map $\phi: T\to G$}}\}.$$

\begin{definition}{\cite[Definition 0.9]{KWZ2}}
\label{xxdef5.1}
Let $G$ be a locally finite graded noetherian algebra. 
Let ${\mathcal P}$ be any numerical invariant that is defined 
for locally finite ${\mathbb N}$-graded noetherian rings 
{\rm{(}}or connected graded noetherian AS regular algebras{\rm{)}}.
The {\it ${\mathcal P}$-concavity} of $G$ is defined to be
\[
c_{\mathcal P}(G):=\inf_{T \in \Phi(G)}\{ {\mathcal P} (T) \}.
\]
If no such $T$ exists, we define $c_{\mathcal P}(G)=\infty$.

When $\mathcal{P} = -\CMreg$ we call 
\[c(G) := c_{-\CMreg}(G)\]
simply the \emph{concavity} of $G$.
\end{definition}

In this subsection we introduce the weighted version of concavity, 
that is, when $\mathcal{P} = -\CMregxi$ in the above definition.

\begin{definition} 
\label{xxdef5.2}
Let $G$ be a locally finite graded noetherian algebra and let 
$\xi \in \mathbb{R}$. The {\it $\xi$-concavity} of $G$ is defined 
to be
\[
c_{\xi}(G) := c_{-\CMregxi}(G).
\]
\end{definition}

Now we prove the following weighted analogues of 
\cite[Theorem 0.10(1) and Proposition 4.1(4)]{KWZ2}, which follow 
from weighted versions of the proofs in \cite{KWZ2}.

\begin{proposition}
\label{xxpro5.3}
\begin{enumerate}
\item[(1)] 
Let $t$ be a commutative indeterminate and assume that 
$\xi\geq \deg t$. Then $c_{\xi}(G[t])=c_{\xi}(G)$.
\item[(2)] 
Let $T$ be a noetherian AS regular algebra. Suppose 
$0\leq \xi\leq 1$. Then 
$$c_{\xi}(T)=-\CMregxi(T).$$
\end{enumerate}
\end{proposition}

\begin{proof}
(1) There is a finite map $G[t] \to G$ given by sending $t$ 
to $0$. Hence, taking $\mathcal{P} = -\CMregxi$ in 
\cite[Proposition 4.1(1)]{KWZ2}, we have 
$c_{\xi}(G[t])\geq c_{\xi}(G)$.

Fix a real number $\epsilon >0$. By definition of $c_{\xi}(G)$, 
there is a noetherian AS regular algebra $T$ of type $(d,\bfl)$ 
and a finite map $\phi: T\to G$ such that 
$-\CMregxi(T)\leq c_{\xi}(G)+\epsilon$. Then $T[t]\to G[t]$ is 
a finite map. Hence,
$$\begin{aligned}
c_{\xi}(G[t])&\leq -\CMregxi(T[t])=-(\xi(d+1)-(\bfl+\deg t))\\
&=-(\xi d-\bfl) -(\xi-\deg t) \leq -(\xi d-\bfl)\\
&=-\CMregxi(T)\leq c_{\xi}(G)+\epsilon.
\end{aligned}
$$
Since $\epsilon$ was arbitrary, we obtain that $c_{\xi}(G[t])
\leq c_{\xi}(G)$. Combined with the previous paragraph, we 
conclude that $c_{\xi}(G[t])= c_{\xi}(G)$.

(2) Fix a noetherian AS regular algebra $T$. Recall that 
$0\leq \xi \leq 1$. Suppose $S$ is any noetherian AS regular 
algebra and suppose $S\to T$ is a finite map. By Theorem 
\ref{xxthm0.4} and the fact that $\Torregxi({}_{S} T)\geq 0$, 
we obtain that
\[\CMregxi(S) = 
\CMregxi({}_S T) - \Torregxi({}_S T) \leq \CMregxi(T)\]
and hence, $- \CMregxi(S) \geq -\CMregxi(T)$. Therefore, 
$c_{\xi}(T) \geq -\CMregxi(T)$. By definition, it is clear that 
$c_{\xi}(T) \leq - \CMregxi(T)$, and so we have equality, as 
desired.
\end{proof}

\subsection{Rate (or rate of growth of homology)}
\label{xxsec5.2} 
We first recall the notion of the rate of growth of homology that 
was introduced by Backelin \cite[p.81]{Ba}.

\begin{definition}
\label{xxdef5.4}
Let $A$ be a connected graded algebra.
The {\it rate of the homology} of $A$ is defined to be
$$\rate(A):=\max\left\{1, \sup_{i\geq 2}
   \{(t^i(_A\Bbbk)-1)/(i-1)\} \right\}.$$
\end{definition}

It is clear that $\rate(A) \geq 1$. If $A$ is a commutative 
finitely generated connected graded algebra, then by 
Corollary~\ref{xxcor5.9}, $\rate(A)$ is finite. We show that 
$\rate(A)$ is finite if and only if $\Torregxi(\Bbbk)$ is 
finite for some $\xi$, and in Proposition~\ref{xxpro5.8} we 
provide a sufficient condition, for the finiteness of 
$\Torregxi(\Bbbk)$ that includes the case when $A$ is commutative.

\begin{question}
\label{xxque5.5}
Suppose that $A$ is connected graded and noetherian (not 
necessarily generated in degree 1).
\begin{enumerate}
\item[(1)]
Is $\rate(A)$ always finite?
\item[(2)]
If the answer to part (1) is no, is there a natural condition 
on $A$ that guarantees the finiteness of $\rate(A)$?
\item[(3)]
If $A$ has a balanced dualizing complex, is $\rate(A)$ finite?
\end{enumerate}
\end{question}

The finiteness of $\rate(A)$ is particularly interesting due to 
a result of Backelin, which we now describe. Suppose that $A$ is 
generated in degree 1. Then $A$ is Koszul if and only if 
$\rate(A)=1$. Backelin proved that the finiteness of $\rate(A)$ 
is related to the Koszul property of the Veronese subrings of 
$A$. Let $d \geq 2$ be an integer and define the 
\emph{$d$th Veronese subring of $A$}:
$$
A^{(d)}:=\bigoplus_{i\geq 0} A_{di}.
$$
In this setting, we regrade so that elements in $A_{di}$ have 
degree $i$.

\begin{theorem}[{\cite[Corollary, p.81]{Ba}}]
\label{xxthm5.6}
Let $A$ be a connected graded algebra generated in degree 1. 
If $d \geq \rate(A)$, then $A^{(d)}$ is Koszul.
\end{theorem}

Earlier, in the context of commutative algebra (and algebraic 
geometry) Mumford proved that if $A$ is a connected graded 
finitely generated commutative algebra, then the Veronese 
subring $A^{(d)}$ is Koszul for $d\gg 0$, see 
\cite[lemma, p. 282]{Mu1} and \cite[Theorem 1]{Mu2}.
Hence, Backelin's result extends Mumford's result to the 
noncommutative setting, with the additional assumptions that 
$A$ is generated in degree $1$ and rate$(A)$ is finite. In 
general, Mumford's result fails in the noncommutative 
setting---see \cite[Corollary 3.2]{StaZ}, an example that is 
not generated in degree 1. In Example \ref{xxex5.11} we show 
for this algebra $A$ and for $\xi \geq 3, \Torregxi(X)<\infty$ 
for all $X\in \D^{\bb}_{\fg}(A\Gr)$. 

\begin{lemma}
\label{xxlem5.7}
Suppose $A$ is connected graded and noetherian. 
\begin{enumerate}
\item[(1)]
Suppose $r:=\rate(A)$ is finite 
Then
$$\Torreg_{(1,r)} (\Bbbk)\leq \max\{0, 1-r, t_1^A(\Bbbk)-r\}.$$
\item[(2)]
Suppose that $a:=\Torreg_{\xi}(\Bbbk)$ is finite for some $\xi$. 
Let $a'=\max\{a, 1-\xi\}$.
Then $\rate(A)\leq \max\{1,a'+2\xi-1\}$.
 \item[(3)] Therefore  $\rate(A)$ is finite if and 
only if $\Torreg_{\xi}(\Bbbk)$ is finite for some $\xi$.
\end{enumerate}
\end{lemma}

\begin{proof}
(1) By definition, for all $i\geq 2$, $t_i^A(\Bbbk)-1\leq r(i-1)$. 
Then $t_i^A(\Bbbk)- r i\leq 1-r$. The assertion follows.

(2) By definition, $t^A_i(\Bbbk)\leq a+i \xi\leq a'+i\xi$ for 
all $i\geq 2$. Then 
$$\begin{aligned}
\sup_{i\geq 2} \{(t^A_i(\Bbbk)-1)/(i-1)\}
&\leq \sup_{i\geq 2} \{(a'+ i\xi-1)/(i-1)\}\\
&=\sup_{n\geq 1}
\{[(a'+ \xi-1)+n\xi]/n\}
=a'+2\xi-1.
\end{aligned}
$$

(3) This is an immediate consequence of parts (1) and (2).
The assertion follows.
\end{proof}

The next proposition provides a criterion for the finiteness of 
$\Torregxi(\Bbbk)$ for some $\xi$. 

\begin{proposition}
\label{xxpro5.8}
Let $A$ be a noetherian connected graded algebra 
and suppose that there is a finite map $\phi: T \to A$ 
where $T$ is a noetherian connected graded algebra of 
finite global dimension. Let 
\[
c:=\max\left \{ t^T_0({}_TA), 
\max_{1\leq s\leq \pdim( {}_TA )} 
\left\{ t^T_s(_TA)/s \right\} \right\}< \infty.
\]
\begin{enumerate}
\item[(1)]
Write $\xi=(1,\xi)$. If $\xi \geq c$, then 
$\Torregxi(_A\Bbbk)<\infty$. 
\item[(2)]
If further $T$ is AS regular and $\xi \geq c$, then 
$\Torregxi(X)<\infty$ for all $X\in \D^{\bb}_{\fg}(A\Gr)$.
\item[(3)]
Suppose $T$ is AS regular. Then $\ASregxi(A)<\infty$ 
for all $\xi\gg 0$.
\end{enumerate}
\end{proposition}

\begin{proof} (1) It is enough to show the assertion for 
$\xi=(1,c)$. By definition of $c$, we have
$$t^T_0(_TA)\leq c 
{\text{ and }} 
t^T_{s}(_TA)\leq cs \qquad {\text{for all $s\geq 1$}}.$$ 
Let 
\[d :=\max_{0\leq s\leq \gldim T} 
\{t^A_s(\Bbbk)- cs\},\]
which is clearly finite. We claim that 
$t^A_j(\Bbbk)-c j\leq d$ for all $j \geq 0$, which is equivalent to 
the main assertion. We prove this by induction. By the definition 
of $d$, the claim holds for all $0 \leq j \leq \gldim T$.

Now assume that $j > \gldim T$. We will use the change of rings 
spectral sequence given in \cite[Theorem 10.60]{Ro}, namely:
$$E^2_{p,q}:=\Tor^A_{p}\left(\Tor^T_q(\Bbbk_T,A),{_AM}\right)
\Longrightarrow \Tor^T_{p+q}(\Bbbk_T, {_TM}).$$
Letting $M = \Bbbk$ in this spectral sequence and assuming the
induction hypothesis that
$$t^A_s(\Bbbk)-c s\leq d$$ 
for all $s \leq j-1$, we see that for $r \geq 2$,
$$\begin{aligned}
\deg E^{r}_{j-r,r-1} & \leq \deg E^{2}_{j-r, r-1} 
\leq \max_{p\leq j-1}\{ t^A_{p}(\Bbbk)+ t^T_{j-1-p}(_TA)\}\\
&\leq \max\{ t^A_{j-1}(\Bbbk)+ t^T_{0}(_TA), 
\max_{0 \leq p\leq j-1}\{ t^A_{p}(\Bbbk)+ t^T_{j-1-p}(_TA)\}\}\\
&\leq \max\{c(j-1)+d+c, \max_{0 \leq p\leq j-1}
\{(c p+d)+(j-1-p) c\}\}\\
&=d +cj.
\end{aligned}
$$
Note that the incoming differentials to $E^r_{j,0}$,
for $r\geq 2$, are all zero, and the outgoing differentials
from $E^r_{j,0}$ land at $E^r_{j-r,r-1}$ with
$\deg E^r_{j-r,r-1}\leq d +cj$. When $j>\gldim T$, 
$E^{\infty}_{j,0}=\Tor^T_j(\Bbbk,\Bbbk)=0$. Hence $E^2_{j,0}$ 
has a filtration such that each subfactor is a submodule of 
some $E^r_{j-r,r-1}$ where $2\leq r\leq j$. Therefore
$$\begin{aligned}
t^A_j(\Bbbk)
&=\deg \Tor^A_j(\Bbbk, \Bbbk)=\deg \Tor^A_j(\Bbbk\otimes_T A, 
\Bbbk)\\
&=\deg E^2_{j,0} \leq \max_{r\geq 2}\{\deg E^r_{j-r,r-1}\}\\
&\leq d +cj.
\end{aligned}
$$
This finishes the inductive step and the proof of the main assertion.

(2) Since $T$ has a balanced dualizing complex, so does $A$, via 
the finite map $\phi:T \to A$. The assertion follows from part 
(1) and Theorem \ref{xxthm0.3}(1).

(3) The assertion follows from part (2) and 
the definition of $\ASregxi(A)$.
\end{proof}

The above proposition shows that if $A$ is finitely generated 
and commutative then there exists a weight $\xi$ such that 
$\Torregxi({}_A \Bbbk)$ is finite; hence by Lemma \ref{xxlem5.7}, 
$\rate(A)$ is finite, as noted earlier.

\begin{corollary}
\label{xxcor5.9}
Let $A$ be a noetherian connected graded algebra generated in 
degree 1 and suppose there is a finite map $T \to A$ where $T$ 
is a noetherian connected graded algebra of finite global 
dimension. Then $\rate(A)$ is finite and hence $A^{(d)}$ is 
Koszul for $d\gg 0$.
\end{corollary}

\begin{proof} The assertion follows from 
Proposition \ref{xxpro5.8}, Lemma \ref{xxlem5.7}, 
and Theorem \ref{xxthm5.6}.
\end{proof}

When $A$ is commutative and finitely generated, then there is 
a surjective map from a polynomial ring to $A$. So 
Corollary \ref{xxcor5.9} recovers Mumford's result 
\cite[Theorem 1]{Mu2}. This motivates the following questions
that are related to Question \ref{xxque5.5}.

\begin{question}
\label{xxque5.10}
Let $A$ be a noetherian connected graded algebra.
\begin{enumerate}
\item[(1)] 
Suppose $A$ is generated in degree $1$. Is then $A^{(d)}$ 
Koszul for $d\gg 0$? 
\end{enumerate}
If the answer to part (1) is no, we further ask
\begin{enumerate}
\item[(2)] 
Suppose $A$ is generated in degree $1$. Is there a natural 
homological condition such that $A^{(d)}$ is Koszul for 
$d\gg 0$?  For example, if $A$ has a balanced dualizing complex 
is $A^{(d)}$ Koszul for $d\gg 0$?
\item[(3)]
Suppose further that $A$ is PI. Is then $A^{(d)}$ Koszul for 
$d\gg 0$? 
\end{enumerate}
If $A$ is not generated in degree $1$, then it is not necessary 
that $A^{(d)}$ is Koszul for $d \gg 0$. We can therefore ask 
part (2) in this setting, namely:
\begin{enumerate}
\item[(4)] 
If $A$ is not necessarily generated in degree $1$, is there a 
natural homological condition which guarantees that $A^{(d)}$ 
is Koszul for $d \gg 0$?
\end{enumerate}
\end{question}

The next example shows that the hypothesis in Proposition 
\ref{xxpro5.8}(2) is sufficient, but not necessary.

\begin{example}
\label{xxex5.11}
Assume that $\Bbbk={\mathbb C}$. Let $U$ be the algebra 
$\Bbbk\langle x,y\rangle/(yx-xy-x^2)$ and let $R=\Bbbk+Uy$. 
The algebra $U$ is noetherian AS regular of global dimension 
two and $R$
is noetherian and generated by $y$ and $xy$. 
It follows from \cite[Theorem 2.3]{StaZ} that there is no 
finite map from a noetherian AS regular algebra $T$ to $R$.

We claim that if $\xi_1\geq 3$ (and write $\xi=(1,\xi_1)$), then 
$\Torregxi(X)<\infty$ for all $X\in \D^{\bb}_{\fg}(R\Gr)$. We 
give a sketch of the proof below.

\medskip

\noindent {\sf Claim 1: Consider $U$ as a left graded $R$-module. 
Then $_RU$ is finitely generated and $\Torregxi(U)<\infty$.}

\medskip

\noindent
{\it Proof of Claim 1:} By \cite[(2.3.1)]{StaZ}, we have a 
short exact sequence
$$0\to Uhx\to R\oplus Rx \to U\to 0$$
where $h=(y^2-2xy)$. Using this we obtain the following
minimal free resolution of the $R$-module $_RU$:
$$\cdots
\to R(-9)\oplus R(-10)
\to R(-6)\oplus R(-7)
\to R(-3)\oplus R(-4)
\to R\oplus R(-1) \to U\to 0$$
which implies that $t^i(_RU)=3i+1$ for all $i\geq 0$ and 
$$\Torregxi(_RU)=
\begin{cases}
\infty & \xi_1<3,\\
1 & \xi_1\geq 3.
\end{cases}
$$
In particular, when $\xi_1\geq 3$, $\Torregxi(_RU)<\infty$.

\medskip

\noindent 
{\sf Claim 2: Suppose $\xi_1\geq 3$. If $M$ is a finitely 
generated graded left $U$-module, then $\Torregxi(_RM)<\infty$.}

\medskip

\noindent
{\it Proof of Claim 2:} Since $U$ is AS regular,
there is a minimal free resolution
$$0\to P_2\to P_1\to P_0\to M\to 0.$$
By Claim 1, $\Torregxi(_RP_i)<\infty$ for $i=0,1,2$
(as we assume $\xi_1\geq 3$). By Lemma \ref{xxlem3.7}, 
$\Torregxi(_RM)<\infty$.

\medskip

\noindent {\sf Claim 3: Suppose $\xi_1\geq 3$. If $M$ is a finitely 
generated graded left $R$-module, then $\Torregxi(_RM)<\infty$.}

\medskip

\noindent
{\it Proof of Claim 3:} We use induction on the 
Krull dimension of $M$. 

\noindent \textbf{Case 1:} Suppose $M$ has Krull dimension 0.
Then $M$ is finite dimensional. If $M$ is 1-dimensional, 
then it is of the form $\Bbbk(n)$, which is an $U$-module, and
the assertion follows from Claim 2. If $\dim M>1$, the 
assertion follows from Lemma \ref{xxlem3.7} and the base case 
$\Bbbk(n)$. The minimal free resolution of the trivial module 
$_R\Bbbk$ (which was computed by Frank Moore):
$$
\cdots \to R(-(3n-1))\oplus R(-(3n-2))\to 
\cdots \to R(-2)\oplus R(-1)\to R\to \Bbbk\to 0.
$$
As a consequence
$$\Torregxi(_R\Bbbk)=
\begin{cases}
\infty & \xi<3,\\
0&\xi_1\geq 3.
\end{cases}
$$

\noindent 
\textbf{Case 2:} Assume $M$ has Krull dimension 1. By noetherian 
induction, we need to consider only  the case when $M$ is 
1-critical. First we recall that 1-critical $U$-modules
are graded shifts of $U/U(ax+by)$ (these are called 
point modules). Since ${\text{Proj}}\;R={\text{Proj}}\; U$,
every 1-critical $R$-module $M$ is a submodule of a 
1-critical $U$-module $N$ such that $N/M$ is finite 
dimensional. By Claim 2, Case 1, and Lemma \ref{xxlem3.7}, 
the assertion follows.

\noindent 
\textbf{Case 3:} Suppose $M$ has Krull dimension 2. By noetherian 
induction, we need to consider only the case when $M$ is 
2-critical. Then $M$ contains a submodule isomorphic to 
$R(n)$, and $M/R(n)$ has Krull dimension 1. So the assertion
follows from Case 2 and Lemma \ref{xxlem3.7}. Combining
these cases we finish the proof of Claim 3.

\medskip

\noindent
{\sf Claim 4: Suppose $\xi_1\geq 3$. If $X$ is in 
$\D^{\bb}_{\fg}(R\Gr)$, then $\Torregxi(X)<\infty$.}

\medskip

\noindent
{\it Proof of Claim 4:} Define the {\it amplitude} of 
a complex $X$ to be
$$\amp(X):=\sup(X)-\inf(X).$$ 
If $\amp(X)=0$, then $X$ is isomorphic to a complex
shift of a module. The claim follows from 
Lemma \ref{xxlem3.1}(2) and Claim 3. If $\amp(X)>0$,
then by truncation, there are complexes $Y$ and $Z$
with smaller amplitude than $\amp(X)$ such that
$$Y\to X\to Z\to Y[1]$$
is a distinguished triangle. The claim follows from 
induction and Lemma \ref{xxlem3.7}(1). 
\end{example}

The following proposition follows from the above example.

\begin{corollary}
\label{xxcor5.12} 
There is a noetherian connected graded algebra $A$ with finite
$\Torregxi(\Bbbk)$ for some $\xi$, but not generated in degree 1, 
such that the Veronese subring $A^{(d)}$ is not Koszul for every
$d\gg 0$.
\end{corollary}

\begin{proof}
Let $A$ be the algebra $R$ in Example \ref{xxex5.11}. By 
\cite[Corollary 3.2]{StaZ}, the Veronese subring $A^{(d)}$ is 
not Koszul for all $d\gg 0$. By Example \ref{xxex5.11},
$\Torregxi(\Bbbk)$ is finite for $\xi=(1,3)$.
\end{proof}

\subsection{Slope}
\label{xxsec5.3}
Another homological invariant that is related to the 
$\xi$-Tor-regularity is the \emph{slope} of a 
graded $A$-module $M$, which was introduced in \cite{ACI1} 
for finitely generated commutative algebras.

\begin{definition}\cite[p.197]{ACI1}
\label{xxdef5.13}
Let $M$ be a graded left $A$-module. The {\it slope} of $M$ is 
defined to be
$${\rm slope}~M := \sup_{i \ge 1} 
\frac{\deg (\Tor^A_i(\Bbbk, M)) - \deg (\Tor^A_0(\Bbbk, M))}{i}.$$
\end{definition}

Following \cite{Ba}, $\rate(A)=\slope_A(A_{\geq 1})$. 

When $A$ is a finitely generated  commutative connected graded 
algebra, then in \cite[Corollary 1.3]{ACI1} it is proved that 
for every finitely generated graded $A$-module $M$, the slope 
of $M$ is finite. The relationship between slope and $\CMreg$ 
and $\Torreg$ was  
discussed in \cite{ACI1}. Also see \cite{ACI2} 
(and the references therein) for the study of $t^A_i(M)$ in the 
commutative setting.

\begin{remark}
\label{xxrem5.14} 
Let $A$ be a connected graded noetherian algebra
and $M$ be a finitely generated graded left $A$-module. 
\begin{enumerate}
\item[(1)]
It is clear that the definition of the $\slope(M)$ makes sense 
in the noncommutative setting, too. 
\item[(2)]
If $s=s(M):=\slope(M)$ is finite, then 
$\Torreg_{s}(M)\leq t^A_0(M)$.
%%%zhang \textcolor{green}{As a consequence, $b_{\Torreg}(M)\leq b_{\Torreg}(_A\Bbbk)
%%%zhang \leq s(\Bbbk)$. [If we do not define $b_{\Torreg}$, then we should delete this.]}
\item[(3)]
If $\Torregxi(M)$ is finite for some $\xi$, then
$\slope(M)\leq \xi+|\Torregxi(M)-t^A_0(M)|$.
\item[(4)]
Combining parts (2) and (3), $\slope(M)$ is finite 
if and only if $\Torregxi(M)$ is finite for some $\xi$.
\item[(5)]
In the setting of Proposition \ref{xxpro5.8}(2),
by part (4), $\slope(M)$ is always finite.
\item[(6)]
$c_{\xi}$ (in Definition \ref{xxdef5.2}), $\rate$, $\slope$, and 
$\Torregxi$ are useful in understanding various homological 
properties, which was demonstrated in \cite{ACI1, ACI2} in the 
commutative case.
\end{enumerate}
\end{remark}

We conclude the paper by asking a final question.
If $A$ is commutative, then by \cite{AP} if 
$\Torreg(\Bbbk) < \infty$, then $\Torreg(\Bbbk) = 0$.
Although this is not true in the noncommutative setting, 
we remark that if $A$ is commutative then $c(A) = 0$.

\begin{question}
\label{xxque5.15}
If $c(A)=0$ and $\Torreg(\Bbbk)<\infty$,  then is
$\Torreg(\Bbbk)=0$ (or equivalently, $A$ is  Koszul)?
\end{question}

\subsection*{Acknowledgments}
The authors thank the referee for his/her very careful reading 
and valuable comments and thank Frank Moore for useful 
conversations on the subject and computational assistance. 
R. Won was partially supported by an AMS--Simons Travel Grant and Simons Foundation grant \#961085.
J.J. Zhang was partially supported by the US National Science 
Foundation (Nos. DMS-1700825, DMS-2001015, and DMS-2302087).

\bibliographystyle{amsalpha}
\bibliography{biblio}

\end{document}